\documentclass[11pt]{amsart}
\usepackage{amssymb,amsmath,epsfig,color}
\newtheorem{theorem}{Theorem}[section]
\newtheorem{lemma}[theorem]{Lemma}
\newtheorem{proposition}[theorem]{Proposition}
\newtheorem{corollary}[theorem]{Corollary}
\newtheorem{conjecture}[theorem]{Conjecture}
\theoremstyle{definition}
\newtheorem{definition}[theorem]{Definition}
\newtheorem{example}[theorem]{Example}
\newtheorem{remark}[theorem]{Remark}
\newtheorem{question}[theorem]{Question}

\def\Susp{\sum}

\def\lk{\mathrm{lk}}

\title{Homotopy type of circle graphs complexes motivated by extreme Khovanov homology}

\author{Jozef H. Przytycki}
\author{Marithania Silvero}						

\begin{document}
\maketitle
\markboth{\hfil{\sc Homotopy type of circle graphs motivated by extreme Khovanov homology}\hfil}
\

\begin{center}
\textbf{Abstract}
\end{center}
\vspace{0.2cm}
It was proven in \cite{GMS} that the extreme Khovanov homology of a link diagram is isomorphic to the reduced (co)homology of the independence simplicial complex obtained from a bipartite circle graph constructed from the diagram. In this paper we conjecture that this simplicial complex is always homotopy equivalent to a wedge of spheres. In particular, its homotopy type, if not contractible, would be a link invariant and it would imply that the extreme Khovanov homology of any link diagram does not contain torsion. We prove the conjecture in many special cases and find it convincing to generalize it to every circle graph (intersection graph of chords in a circle). In particular, we prove it for the families of cactus, outerplanar, permutation and non-nested graphs. Conversely, we also give a method for constructing a permutation graph whose independence simplicial complex is homotopy equivalent to any given finite wedge of spheres. We also present some combinatorial results on the homotopy type of finite simplicial complexes and a theorem generalizing previous results by Csorba, Nagel and Reiner, Jonsson and Barmak. We study the implications of our results to Knot Theory; more precisely, we compute the real-extreme Khovanov homology of torus links $T(3,q)$ and obtain examples of $H$-thick knots whose extreme Khovanov homology groups are separated either by one or two gaps as long as desired.

\tableofcontents

%%%%%%%%%%%%%%%%%%%%%%%%%%%%%%%%%%%%%%%%%%%%%

\section{Introduction}

This paper is motivated by a result in \cite{GMS} where the authors showed that the extreme Khovanov homology associated to a link diagram is isomorphic to the (co)homology of the independence simplicial complex of a special graph constructed from the diagram (the so-called Lando graph). In this paper we conjecture that this simplicial complex is homotopy equivalent to a wedge of spheres. More generally, we state the following: \\

\noindent\textbf{Conjecture}
\begin{enumerate}
\item[(1)] The independence simplicial complex associated to a circle graph is homotopy equivalent to a wedge of spheres.
\item[(2)] In particular, the independence simplicial complex associated to a bipartite circle graph (Lando graph) is homotopy equivalent to a wedge of spheres.
\item[(3)] In particular, the extreme Khovanov homology of any link diagram is torsion-free. \\
\end{enumerate}

The independence complex of a graph belonging to the families of paths, trees and cycle graphs is known to be homotopy equivalent to a wedge of spheres \cite{kozlov}. We give support to the above conjecture by extending these results to many special cases, e.g. cactus, outerplanar, permutation and non-nested graphs. Conversely, we also show that the homotopy type of any finite connected wedge of spheres can be realized as the independence complex of a circle graph.

The results in this paper, which has been conceived from a combinatorial point of view, have several implications in Knot Theory. More precisely, we study the extreme Khovanov homology of torus links, and show how the study of the homotopy type of Khovanov homology leads to new results and conjectures. These ideas were well expressed by B. Everitt and P. Turner:\emph{``Another approach is to interpret the existing constructions of Khovanov homology in homotopy theoretic terms. By placing the constructions into a homotopy setting one makes Khovanov homology amenable to the methods and techniques of homotopy theory''} \cite{EveritTurner}.

We also noticed connections between our work involving circle graphs and pseudo-knots (planar and spacial) formed by a secondary structure of RNA. Despite being worth of careful exploring, this topic lies outside the scope of this paper (see, for example, \cite{Pseud1, Pseud2} for definitions of pseudo-knots and their relation to chord diagrams).

The plan of the paper is as follows. In the second section we review basic definitions and present the conjecture we will deal with throughout the paper. In the third section we describe the classical idea of building the simplicial complex ``cone by cone" (in essence the cell decomposition of the complex), and present the first results involving the families of cactus and outerplanar graphs. In Sections 4 and 5 we prove our conjecture for the family of permutation graphs and non-nested circle graphs, respectively. In the sixth section we prove a general theorem on independence complexes which generalizes results by Csorba, Nagel and Reiner, Jonsson, and Barmak. Finally, in Section 7 we show some applications of our work to Knot Theory, namely we compute the extreme Khovanov homology of torus links $T(3,q)$ and construct two families of $H$-thick knots having two and three non-trivial extreme Khovanov homology groups separated by gaps as long as desired.

\section{Preliminaries} \label{Prel}

In this section we review some well-known concepts and provide basic tools that will be useful throughout this paper. We based our exposition on \cite{Brown}, \cite{Chmutbook}, \cite{Hatcher} and \cite{Jonssonbook}.

\subsection{Wedges, joins and independence simplicial complexes}

\begin{definition}
An abstract simplicial complex ${\mathcal K} = (V,P)$ consists of a pair of sets $V = V({\mathcal K})$ and $P({\mathcal K}) = P \subset 2^V$, called set of vertices and set of simplexes of ${\mathcal K}$ respectively. The elements of $P$ are finite subsets of $V$, include all one-element subsets, and if $s' \subset s \in P$ then also $s'\in P$ (that is, a subsimplex of a simplex is a simplex). A simplex containing $n+1$ vertices is an $n$-dimensional simplex, or succinctly, $n$-simplex. We define $dim ({\mathcal K})$ as the maximal dimension of a simplex in ${\mathcal K}$ (may be $\infty$ if there is no bound). A simplicial complex is called a {\it flag complex} if it has the property that every pairwise connected set of vertices forms a simplex.
\end{definition}

In this paper we work exclusively with finite simplicial complexes and assume that they contain the empty set as a simplex of dimension -1. A simplicial complex has a natural geometric realization and we often use both topological and combinatorial languages with no distinction. Moreover, to avoid cumbersome notation, we will refer to a simplicial complex as its set of maximal simplexes (facets), and write $v_0 v_1 \ldots v_n$ instead of $\{v_0, v_1, \ldots, v_n\}$ for each simplex. For example, Figure \ref{mhexagon} shows the geometric realization of the abstract $2$-dimensional simplicial complex $\mathcal{K} = \{14, 25, 36, 135, 246\}$.

Given a graph $G$, we define its independence simplicial complex $I_G$ first for a loopless graph and then we extend the definition for the case of graphs with loops. This extended definition allows a shorter description in some of the statements of the following sections.

\begin{figure}
\centering
\includegraphics[width = 8.5cm]{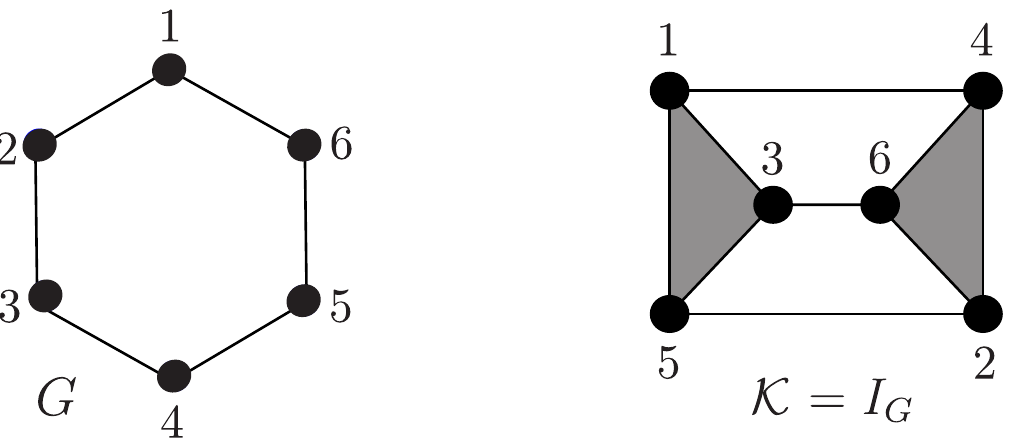}
\caption{\small{The cycle graph $G = C_6$ and its associated independence simplicial complex $\mathcal{K} = I_G$.}}
\label{mhexagon}
\end{figure}

\begin{definition} \label{defindep}
\begin{enumerate}
\item[(1)] Let $G$ be a loopless graph. Then its independence simplicial complex $I_G$ is a complex whose set of vertices is the same as the set of vertices of $G$ and $\sigma = (v_0 v_1 \ldots v_n)$ is a simplex in $I_G$ if and only if the vertices $v_0, v_1, \ldots, v_n$ are independent in $G$ (that is, there are not edges in $G$ between these vertices). Notice that the 1-skeleton of  $I_G$ is $G^c$, the complement graph of $G$; then $I_G$ can also be thought as the simplicial complex of cliques in $G^c$.
\item[(2)] Let $e$ be a loop in the graph $G$ based in the vertex $v$. Then we define $I_G = I_{G-v}$. This definition is justified by the fact that $v$ is connected by an edge with itself in $G$, hence it is not independent and it cannot be a vertex of any simplex in $I_G$.
\end{enumerate}
\end{definition}

Note that it follows from Definition \ref{defindep} that $I_G$ is a flag complex for any graph $G$. Figure \ref{mhexagon} shows a hexagon and its associated independence simplicial complex, which is homotopy equivalent (denoted by $\sim_h$) to the wedge of $S^1$ and $S^1$ (see next definition).

\begin{definition} \label{defwedge} Let $\mathcal{K}_1$ and $\mathcal{K}_2$ be two simplicial complexes.
\begin{enumerate}
\item[(1)] Let $v_i \in V(\mathcal{K}_i)$ be a distinguished vertex called base point, for $i = 1,2$. The wedge (product) of $K_1$ and $K_2$, $(K_1,v_1)\vee (K_2,v_2)$, is a simplicial complex obtained by identifying $v_1$ and $v_2$. The homotopy type of the resulting simplicial complex is preserved under change of the base points $v_1$ and $v_2$ inside the connected components in the original complexes. In particular, for connected simplicial complexes the wedge is well defined up to homotopy equivalence without specifying base points.
\item[(2)] The join of $\mathcal{K}_1$ and $\mathcal{K}_2$, $\mathcal{K}_1 * \mathcal{K}_2$, is a simplicial complex with vertex set $V(\mathcal{K}_1*\mathcal{K}_2) = V(\mathcal{K}_1)\sqcup V(\mathcal{K}_2)$ and simplexes $s \in P(\mathcal{K}_1*\mathcal{K}_2)$, with $s = s_1\sqcup s_2$ for $s_i \in P(\mathcal{K}_i)$.
\item[(3)] If $V(\mathcal{K}_2) = v$, then the join $\mathcal{K}_1 * v$ is called a cone with apex $v$ and base $\mathcal{K}_1$ (sometimes denoted by $cone(v, \mathcal{K}_1)$).
\item[(4)] If $\mathcal{K}_2 = S^0$ (a complex whose topological realization consists of two isolated vertices), then $\mathcal{K}_1 * \mathcal{K}_2$ is called \emph{suspension} of $\mathcal{K}_1$ and denoted by $\Sigma \mathcal{K}_1$. \end{enumerate}
\end{definition}

\begin{remark}\label{remm}
It follows from Definition \ref{defwedge} that the independence complex of the disjoint union of two graphs is equal to the join of their independence simplicial complexes, that is, $$I_{G_1 \sqcup G_2} = I_{G_1} * I_{G_2}.$$
\end{remark}

Next we list several basic properties of wedge and join operations. In particular, we notice that the set of connected finite simplicial complexes (up to homotopy equivalence) with operations $\vee$ and $*$ form, after some modifications, a commutative semiring (compare \cite{Gla}).\footnote{Recall that a semiring
is a set $X$ with two binary operations $+$ and $\cdot$ and two constants $0$ and $1$ such that:\\
(1) $(X,+,0)$ is a commutative monoid with neutral element $0$,\\
(2) $(X,\, \cdot\, , 1)$ is a  monoid  with neutral element $1$, \\
(3) Multiplication is distributive with respect to addition, that is $(a+b) \cdot c = (a \cdot c) + (b \cdot c)$ and $c \cdot (a+b) = (c \cdot a) + (c \cdot b)$,\\
(4) $0\cdot a = 0 = a \cdot 0$.}

\begin{proposition} \label{propwedge} \
\begin{enumerate}
\item [(1)] $S^m*S^n= S^{m+n+1}$, where $S^k$ is the sphere of dimension $k$.
\item [(2)] Let $\mathcal{K}_i \sim_h \mathcal{K}_i'$ for $i=1,2$. Then $\mathcal{K}_1 * \mathcal{K}_2 \sim_h \mathcal{K}_1' * \mathcal{K}_2'$ and  $\mathcal{K}_1 \vee \mathcal{K}_2 \sim_h \mathcal{K}_1' \vee \mathcal{K}_2'$ (basepoints should be chosen coherently).
\item [(3)] $(\mathfrak{K},*)$ is a commutative monoid. Here $\mathfrak{K}$ is a set of finite simplicial complexes with neutral element $\emptyset$ ($\mathcal{K} * \emptyset = \mathcal{K} = \emptyset * \mathcal{K}$). Its elements are considered up to homotopy equivalence. In light of $(1)$ the empty set can be called \emph{a sphere of dimension $-1$} and denoted by $S^{-1}$, thus $\Sigma \emptyset = S^0$. Notice that if $\mathcal{K}$ is not empty, then $\Sigma \mathcal{K}$ is a connected simplicial complex.
\item [(4)] For connected simplicial complexes, $(\mathfrak{K}-\emptyset,\vee)$ is a commutative monoid with a one-element simplicial complex $b$ as the neutral element. We can extend $({\mathfrak K}- \emptyset,\vee)$ to $(\hat{\mathfrak K},\vee)$ by allowing, formally, wedge of any simplicial complex with empty sets:  $\mathcal{K}\vee\emptyset \vee \ldots\vee \emptyset$.
\item [(5)]  For connected simplicial complexes, $*$ is distributive with respect to $\vee$: $$(\mathcal{K}_1\vee \mathcal{K}_2)*K_3 \sim_h (\mathcal{K}_1*\mathcal{K}_3)\vee (\mathcal{K}_2*\mathcal{K}_3).$$
\item [(6)] If $b$ is the one-element simplicial complex, then for any simplicial complex $\mathcal{K}$ it holds that $\mathcal{K}*b$ is contractible ($\mathcal{K}*b\sim_h b$).
\end{enumerate}

By properties (2)-(6) we conclude that $(\hat{\mathfrak{K}},\vee,*)$ is a commutative semiring.

\begin{enumerate}
\item[(7)] The suspension of a wedge of spheres is homotopy equivalent to a wedge of spheres:
$$\Susp\left( S^{i_1}\vee S^{i_2} \vee \ldots \vee S^{i_n}\right) \sim_h S^{i_1+1}\vee S^{i_2+1} \vee \ldots \vee S^{i_n+1}.$$
\item[(8)] For any simplicial complexes $\mathcal{K}_1, \mathcal{K}_2$ it holds $$\Susp (\mathcal{K}_1\times \mathcal{K}_2)\sim_h \mathcal{K}_1*\mathcal{K}_2 \vee \Susp \mathcal{K}_1 \vee \Susp \mathcal{K}_2.$$ In particular, for a torus $T^2=S^1\times S^1$ it follows $\Sigma T^2 \sim_h S^3\vee S^2 \vee S^2$. This example is important for us because $T^2$ is not a wedge of spheres but its suspension is so.\footnote{It is worth here to mention a classical Theorem of Cannon \cite{Can}, stating that the double suspension of a 3-dimensional homological sphere, e.g. Poincar{\'e} sphere, is $S^5$.}
\end{enumerate}
\end{proposition}

\subsection{Cone construction}\label{seccone}

Given a simplicial complex $\mathcal{K} = (V,P)$, for each vertex $v \in V$ we define
$$
\begin{aligned}
lk_\mathcal{K}(v)&: =\{ \sigma - v : v \in \sigma \in P\},\\
st_\mathcal{K}(v)&: = \{\sigma \in P\ | \ \sigma \cup v \in P\}.
\end{aligned}
$$

It follows from definitions
$$
\mathcal{K} \, = \, (\mathcal{K}-v) \cup cone (v, lk_\mathcal{K}(v)) \, = \, (\mathcal{K}-v) \cup st_\mathcal{K}(v)
$$

Hence, if the inclusion $\lk_\mathcal{K}(v) \hookrightarrow \mathcal{K} - v$ is null-homotopic, then
$$
\mathcal{K} \sim_h (\mathcal{K}-v) \vee \Susp \lk_\mathcal{K}(v).
$$

By considering the particular case $\mathcal{K} = I_G$, we get the following result.

\begin{proposition}\label{Reiner}
For any graph $G$
$$I_G \sim_h  I_{G - v} \cup (v * I_{G-st_G(v)}).\footnote{This formula can be interpreted as a lift of a skein relation at the crossing leading to the vertex $v$ (we review the relation vertex-crossing in Subsection \ref{subseclando}). We do not use this observation in the paper but one should keep it in mind together with the reflection of Everitt and Turner cited in the Introduction.}$$
Moreover, if $I_{G-st_G(v)}$ is contractible in $I_{G-v}$ then
$$I_G \sim_h  I_{G - v} \vee \Susp I_{G-st_G(v)}.$$
\end{proposition}

\begin{proof}
The proof follows from the facts that $I_G - v = I_{G-v}$ and $lk_{I_G}(v) = I_{G-st_G(v)}$.
\end{proof}

\begin{figure}
\centering
\includegraphics[width = 12cm]{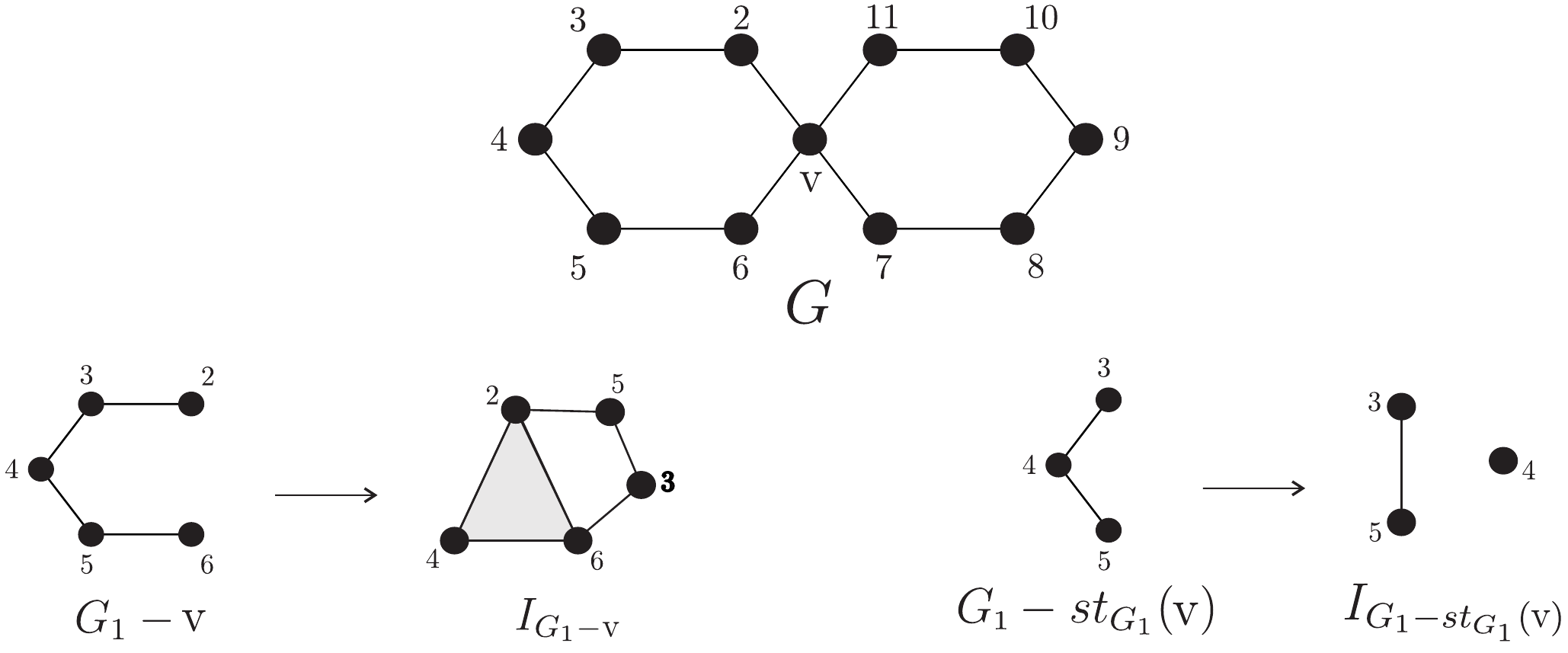}
\caption{\small{The graph $G$, the subgraphs $G_1 - v$, $G_1 - st_{G_1}(v)$ and the topological realization of the complexes $I_{G_1-v}$ and $I_{G_1 - st_{G_1}(v)}$ are shown, illustrating Example \ref{exwedge}. For $G_2$ the computations are analogous.}}
\label{mwedgehex}
\end{figure}

\begin{example}\label{exwedge}
Consider the circle graph $G = G_1 \vee G_2$, where $G_1 = G_2$ are six-cycles (Figure \ref{mwedgehex}). Let $v$ be the wedge vertex identifying base points in the original hexagons. On the one hand, $G-v$ is the disjoint union of two paths of length four, $L_4^1 \sqcup L_4^2$. From Figure \ref{mwedgehex} it is clear that $I_{L_4} \sim_h S^1$. Hence, by Remark \ref{remm}, $I_{G-v}$ is homotopy equivalent to $S^1 * S^1 = S^3$. On the other hand, since $G-st_G(v)$ is the disjoint union of two paths of length two, $I_{G-st_G(v)}$ is homotopy equivalent to $S^0 * S^0 = S^1$. Consequently, $I_{G-st_G(v)}$ is contractible in $I_{G-v}$ and by Proposition \ref{Reiner} $$I_G \sim_h S^3 \vee \Susp S^1 \sim_h S^3 \vee S^2.$$
\end{example}

\subsection{Chord diagrams and circle graphs}

\begin{definition} \
\begin{enumerate}
  \item [(1)] A chord diagram $\mathcal{C}$ is a circle together with a finite set of chords with disjoint boundary points.
  \item [(2)] The circle graph $G$ associated to the chord diagram $\mathcal{C}$ is the intersection graph of its chords, that is, the simple graph constructed by associating a vertex to each chord in $\mathcal{C}$ and connecting two vertices by an edge in $G$ if the corresponding chords intersect. To avoid cumbersome notation, we keep the same name for both the vertex in the graph and its associated chord.
  \item [(3)] A bipartite circle graph is called Lando graph. The chords belonging to a chord diagram leading to a Lando graph can be partitioned into two parts, one of them placed inside and the other one outside the circle, so there are no intersections between the chords. See $\mathcal{C}'$ in Figure \ref{hexagon}.
\end{enumerate}
\end{definition}

\begin{figure}
\centering
\includegraphics[width = 8.5cm]{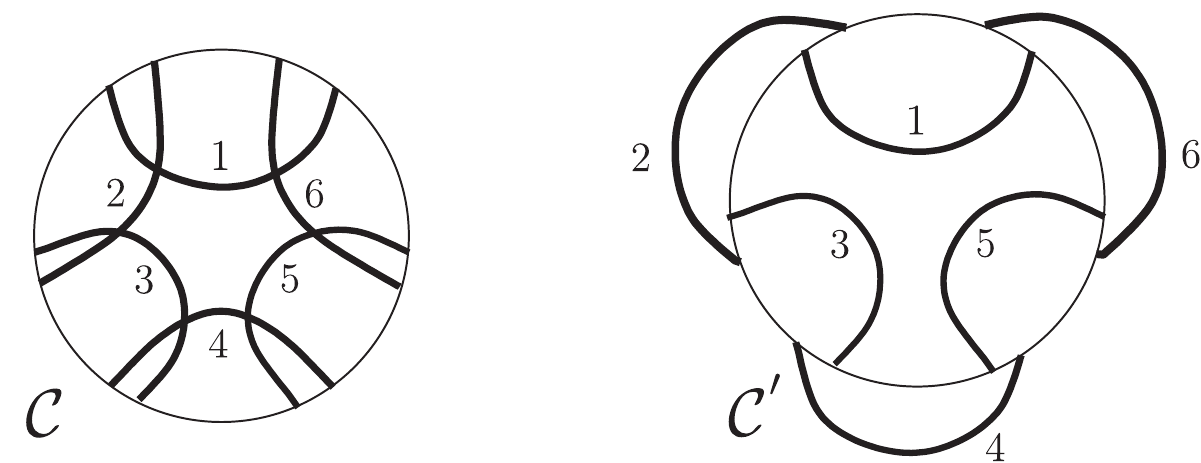}
\caption{\small{Two ways of describing a bipartite chord diagram, $\mathcal{C} \equiv \mathcal{C}'$.}}
\label{hexagon}
\end{figure}

The hexagon in Figure \ref{mhexagon} is the Lando graph associated to the chord diagrams depicted in Figure \ref{hexagon}.

Different chord diagrams may lead to the same circle graph.\footnote{Chmutov and Lando proved that two chord diagrams have the same circle graph if and only if they are related by a sequence of elementary modifications called {\it mutations} \cite{Chmutant}.} Moreover, not every graph is a circle graph, that is, there exist graphs that cannot be represented as intersection graphs of associated chord diagrams.\footnote{In \cite{Bou} Bouchet gives a complete characterization of circle graphs by showing a minimal set of obstructions.}

The following result is useful when constructing a circle graph from simpler pieces.

\begin{lemma}\label{lemacirclewedge}
If $G_1$ and $G_2$ are circle graphs, then $G_1 \vee G_2$ is a circle graph. Moreover, if $G_1$ and $G_2$ are bipartite, hence $G_1 \vee G_2$ is so.
\end{lemma}

\begin{proof}
Let $v_1\in G_1$ and $v_2 \in G_2$ be the base points used to construct $G_1 \vee G_2$. Starting from the two chord diagrams associated to $G_1$ and $G_2$, Figure \ref{circlewedge} represents the chord diagram corresponding to $G_1 \vee G_2$, where $v$ corresponds to the wedge vertex.
\end{proof}

\subsection{Wedge of spheres conjecture}
In this paper we propose the following conjecture:

\begin{conjecture}\label{conj} \
\begin{enumerate}
\item[(1)] The independence simplicial complex associated to a circle graph is homotopy equivalent to a wedge of spheres.\footnote{We consider a contractible set to be homotopy equivalent to an empty wedge of spheres.}
\item[(2)] In particular, the independence simplicial complex associated to a bipartite circle (Lando) graph is homotopy equivalent to a wedge of spheres.
\item[(3)] In particular, the (co)homology groups of the independence simplicial complex of a Lando graph have no torsion. This, by \cite{GMS}, implies that the extreme Khovanov homology of any link diagram is torsion-free.
\end{enumerate}
\end{conjecture}

\begin{figure}
\centering
\includegraphics[width = 12cm]{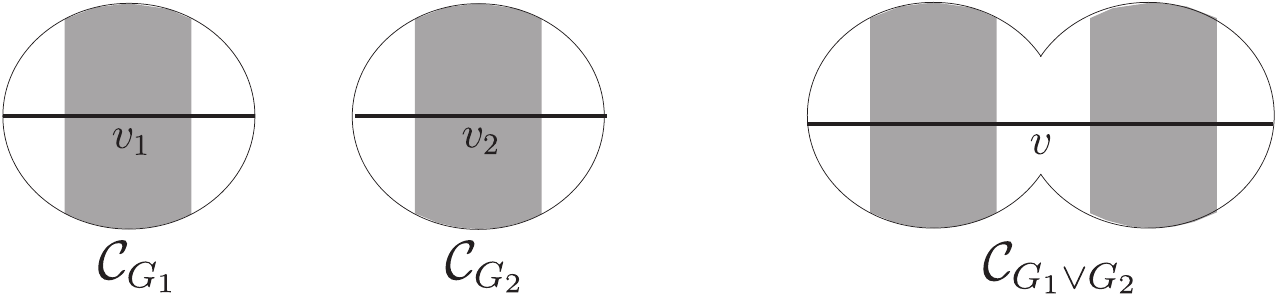}
\caption{\small{The construction of the chord diagram associated to $G_1 \vee G_2$ is shown. The shaded areas represent the rest of the chords in $G_1$ and $G_2$.}}
\label{circlewedge}
\end{figure}

This paper is motivated by the study of extreme Khovanov homology of link diagrams, which was proven in \cite{GMS} to coincide with the reduced cohomology of the independence simplicial complex associated to the Lando graph constructed from the link diagram. This is the reason for considering those particularizations of the general conjecture. We review the definition of Khovanov homology and its relations with bipartite circle graphs in Section \ref{sec7}.

\begin{remark}\label{rechum}
We were informed by Sergei Chmutov on his work on independence complexes of bipartite circle graphs and the talk he gave at Knots in Washington XXI \cite{Chmutconf}, where he conjectured that the independence complex of a bipartite circle graph is homotopy equivalent to a wedge of spheres of the same dimension. Eric Babson sent him a counterexample, the chord diagram shown in Figure \ref{exbabson}(3) leading to a bipartite circle graph of $17$ vertices and $20$ edges, whose independence complex is homotopy equivalent to $S^4 \vee S^5$ \cite{Bab}.

\begin{figure}
\centering
\includegraphics[width = 11cm]{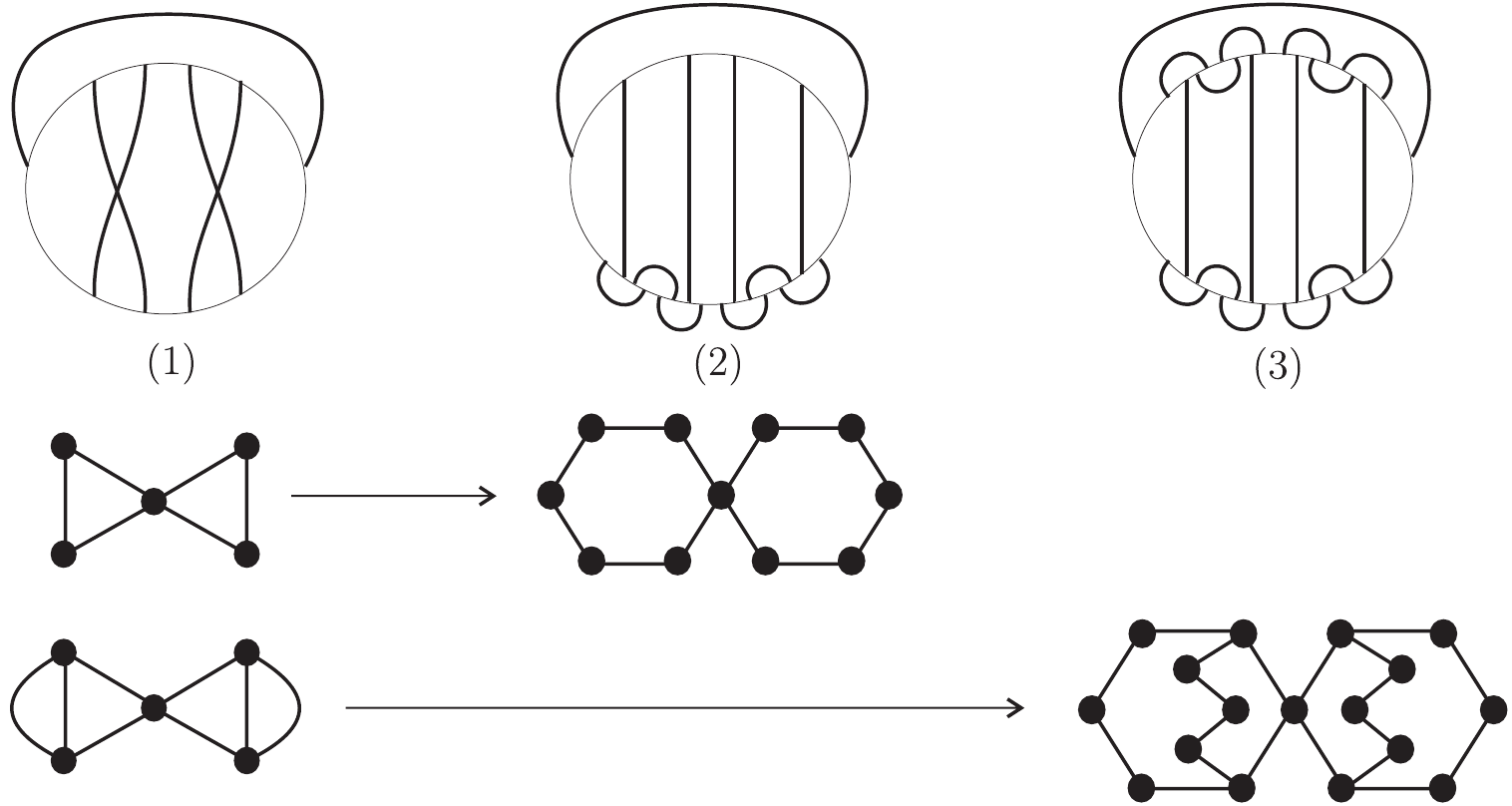}
\caption{\small{We illustrate Remark \ref{rechum} by showing the chord diagrams corresponding to a wedge of two triangles (1), a wedge of two hexagons (2) and the example given by Babson (3) and their associated graphs.}}
\label{exbabson}
\end{figure}

Example \ref{exwedge}, from \cite{GMS}, is simpler (Figure \ref{exbabson}(2)). Both examples can be obtained from a simpler non-bipartite circle graph $G$ consisting on a wedge of two triangles, with $I_G \sim_h S^0 \vee S^1$, by using a result by Csorba stating that if an edge in $G$ is subdivided into four intervals to get the graph $H$, then $I_H \sim_h \Sigma I_G$ \cite{Csorba}. Figure \ref{exbabson} (together with Remark \ref{remcsorba}) illustrates the relations between the three previous examples. We discuss generalizations of Csorba result in Section \ref{sec6}.
\end{remark}

%%%%%%%%%%%%%%%%%%%%%%%%%%%%%%%%%%%%%%%%%%%%%%%%%%%%%%%%%%%%%%%%%%%%%%%%%%%%%%%%%%%%%%%%%
%%%%%%%%%%%%%%%%%%%%%%%%%%%%%%%%%%%%%%%%%%%%%%%%%%%%%%%%%%%%%%%%%%%%%%%%%%%%%%%%%%%%%%%%%%

\section{Domination Lemma and its corollaries}\label{sec3}

In this section we introduce some results that will be useful throughout the paper. We start by reviewing the homotopy type of the independence complexes of paths, trees, and cycles and we prove Conjecture \ref{conj} for the families of cactus and outerplanar graphs. In order to simplify notation we delete $\mathcal{K}$ from $lk_\mathcal{K}(v)$ and $st_\mathcal{K}(v)$ when it is clear from context which simplicial complex is considered.

\begin{definition}
Let $v,w$ be vertices of a graph $G$. We say that $v$ dominates $w$ if \, $lk_G(w) \cup w \cup v \subseteq lk_G(v) \cup v \cup w$.
\end{definition}

\begin{lemma}\label{dominationlemma}[Domination lemma] Let $v, w$ be two vertices of a graph $G$ such that $v$ dominates $w$.
\begin{enumerate}
\item[(1)] \rm{\rm{\cite{Csorba}}} If $v$ and $w$ are not connected by an edge in $G$, then $I_G$ is homotopy equivalent to $I_{G-v}$.
\item[(2)] If $v$ and $w$ are connected by an edge in $G$, then $I_G$ is homotopy equivalent to $I_{G-v} \vee \Sigma I_{G-st(v)}$.
\end{enumerate}

\begin{proof}
(1) By Proposition \ref{Reiner} it suffices to show that $I_{G-st(v)}$ is contractible, which is the case, as $I_{G-st(v)} \sim_h I_{G-st(v)-w}*w$.

(2) By Proposition \ref{Reiner} it suffices to show that $I_{G-st(v)}$ is contractible in $I_{G-v}$. It is the case, since $I_{G-st(v)}$ is embedded in the cone $I_{G-st(w)}*w$ and this cone is embedded in $I_{G-v}$.
\end{proof}
\end{lemma}

The following result is a particular case of \emph{Domination lemma} which is useful when the graph $G$ has a leaf. The only vertex adjacent to a leaf will be called \emph{preleaf}.

\begin{corollary}\label{addingleave}
Let $w$ be a leaf of a graph $G$, and let $v$ be its associated preleaf. Then $$I_G \sim_h \Susp I_{G - st_G(v)}.$$
\end{corollary}

\begin{proof}
Since $v$ dominates $w$ and they are neighbors, the result holds by applying \emph{Domination lemma} together with the fact that $I_{G-v}$ is contractible (since $w$ is isolated in $G-v$).
\end{proof}

\begin{corollary} {\rm{\cite{kozlov}}} \label{path}
Let $L_n$ be the $n$-path, that is, the graph consisting of $n+1$ vertices joined by $n$ edges. Then
$$I_{L_n} \sim_h
\left\{
  \begin{array}{cll}
     * & & \mbox{if } n=3k, \\
     S^k & & \mbox{if } n=3k+1, \, 3k+2.
  \end{array}
  \right.
  $$
\end{corollary}

\begin{corollary}\label{pathwithL3}
Let $G$ be a graph containing two vertices of degree one, $w_1$ and $w_2$, connected by a path of length 3. Then $I_G$ is contractible.
\end{corollary}

\begin{proof}
Let $w_1$, $v_1$, $v_2$ and $w_2$ be the sequence of vertices constituting the path of length 3 in the statement, disposed as shown in Figure \ref{path3}. Applying Corollary \ref{addingleave} taking $w_1$ as leaf-vertex leads to $I_G \sim_h \Susp \left(I_{G-st_G(v_1)}\right) \sim_h \Susp \left(I_{G-st_G(v_1)-w_2} * w_2\right)$, which is contractible.
\end{proof}

\begin{figure}
\centering
\includegraphics[width = 4.7cm]{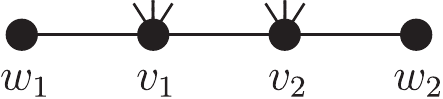}
\caption{\small{The path illustrating Corollary \ref{pathwithL3}.}}
\label{path3}
\end{figure}

\begin{figure}
\centering
\includegraphics[width = 9cm]{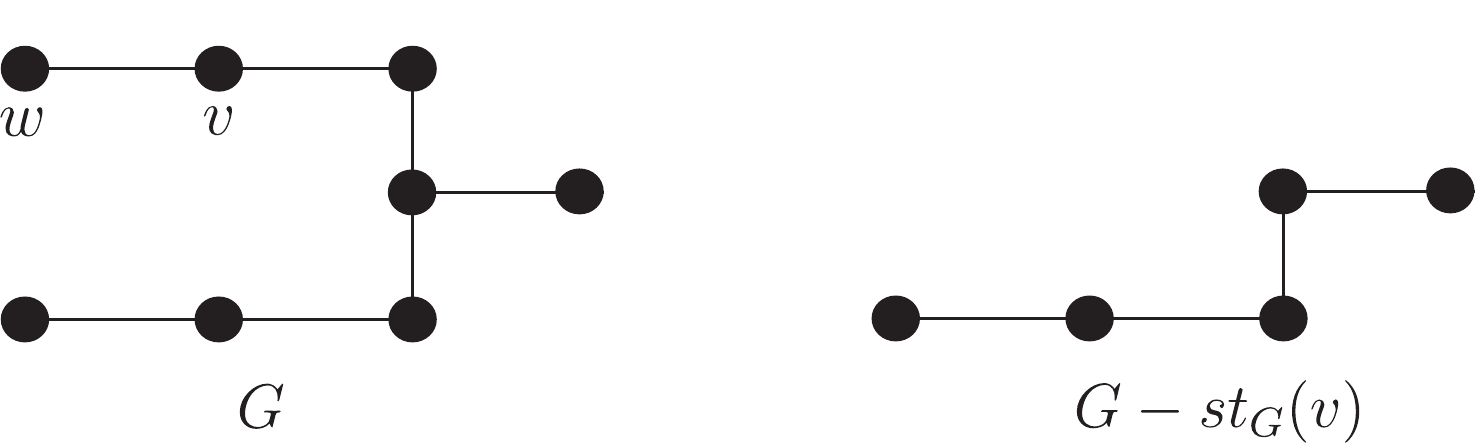}
\caption{\small{The graphs $G$ and $G-st_G(v)$ illustrating Remark \ref{reml6nocont}.}}
\label{l6nocont}
\end{figure}

\begin{remark}\label{reml6nocont}
Note that Corollary \ref{pathwithL3} cannot be extended to the case of a graph containing a path of length 6. The graph $G$ shown in Figure \ref{l6nocont} is such an example, since by Corollary \ref{addingleave} $$I_G \sim_h \Sigma (I_{G-st_G(v)}) = \Sigma I_{L_4} \sim_h \Sigma S^1 = S^2.$$
\end{remark}

\begin{corollary} {\rm{\cite{kozlov}}} \label{proptree}
Let $T$ be a tree, that is, a connected graph with no cycles. Then $I_G$ is either contractible or homotopy equivalent to a sphere.
\end{corollary}

\begin{proposition}\label{treenotdiv3} Let $T$ be a tree.
\begin{enumerate}
  \item [(1)] If $T$ does not contain paths of length divisible by 3 between its leaves, then $I_T$ is homotopy equivalent to a sphere.
  \item [(2)] If $T$ contains a vertex $x$ such that the length of every path from $x$ to a leaf is divisible by 3, then $I_T$ is contractible.
\end{enumerate}
\end{proposition}

\begin{proof}
(1) We proceed by induction on the number $n$ of vertices not being leaves in the graph. The base cases when $n = 0$ and $n=1$ hold, as $I_{\emptyset} = S^{-1}$ and the independence complex of the star graph with $n$ rays is homotopy equivalent to $S^0$.

Now assume that the statement holds when the number of vertices which are not leaves is smaller than $n$. Let $T$ be a tree with no paths of length divisible by 3 connecting two of its leaves, and having $n$ vertices which are not leaves, $n >1$. Let $w$ be one of its leaves and $v$ its associated preleaf. By Corollary \ref{addingleave} $I_T \sim_h \Susp I_{T-st(v)}$. Therefore, by inductive hypothesis it suffices to show that the connected components of $T-st(v)$ do not contain paths of length divisible by 3 between its leaves.

Let $P$ be a path in $T-st(v)$ connecting leaves $w_1$ and $w_2$. If $w_1$ and $w_2$ were leaves in $T$, then by the hypothesis in the statement the length of $P$ is not divisible by 3. As either $w_1$ or $w_2$ is a leaf in $T$, it remains to study the case when exactly one of them is so. In this case, there exists a path in $T$ connecting $P$ with $v_1$, $v$ and $w$, where $v_1 \in lk_T(v)$, and therefore the length of $P$ is not a multiple of 3. This completes the proof of the first part.

(2) As before we proceed by induction, this time on the number of edges of the graph. The base case when there are no edges is trivial. Suppose now that the statement holds when the graph has less than $n$ edges, $n >0$. Let $T$ be a tree containing $n$ edges and let $w$ be a leaf in $T$ with associated preleaf $v$. Note that by hypothesis $x \notin st(v)$, as the distance between $x$ and any leaf is at least 3. Let $T'$ be the connected component of $T-st(v)$ containing $x$ (it satisfies the condition that the length of every path connecting $x$ to a leaf in $T'$ is divisible by 3). By Corollary \ref{addingleave} it holds that $I_T \sim_h \Susp I_{T-st(v)}$. As $T'$ is contractible by inductive hypothesis, the independence complex of $I_T$ is so.
\end{proof}

\begin{proposition} {\rm{\cite{kozlov}}} \label{propcycle}
Let $C_n$ be the cycle graph of order $n$, that is, the $n$-gon. Then
$$I_{C_n} \sim_h
\left\{
  \begin{array}{lll}
     S^{k-1}  & & \mbox{if } n=3k \pm 1, \\
     S^{k-1} \vee S^{k-1} & & \mbox{if } n=3k.
  \end{array}
  \right.
  $$
\end{proposition}

The next result by \cite{Csorba} will be generalized in Theorem \ref{globaltheo} and Corollary \ref{lastcor}.

\begin{theorem}\label{csorba} \emph{\cite{Csorba}}
Let $H$ be a graph. Let $G$ be the graph constructed from $H$ by subdividing one of its edges into four intervals. Then $I_G \sim_h \Susp I_H$.
\end{theorem}

\subsection{Cactus graphs}

\begin{definition}
A cactus graph is a connected graph in which any two cycles have at most one vertex in common. It may be defined constructively as a finite number of wedges of trees and cycle graphs.
\end{definition}

\begin{proposition}\label{cactusprop} Let $G$ be a cactus graph. Then:
\begin{enumerate}
\item [(1)] $G$ is a circle graph.
\item [(2)] The independence complex of $G$, $I_G$, is homotopy equivalent to a wedge of spheres.
\end{enumerate}
\end{proposition}

\begin{proof}
(1) It follows from Lemma \ref{lemacirclewedge} and the fact that trees and cycles are circle graphs.

(2) We proceed by induction on the number of vertices in the graph. When $G$ is empty then $I_G = S^{-1}$, and when it contains exactly one vertex then $I_G$ is contractible. Now suppose that the statement holds for any cactus graph with less than $n$ vertices. Let $G$ be a cactus graph with $n$ vertices, $n>1$. We discuss the different cases:

If $G$ contains at least one leaf with associated preleaf $v$, then by Corollary~\ref{addingleave} it holds that $I_G \sim_h I_{G-st(v)}$, and the inductive hypothesis completes the proof.

\begin{figure}
\centering
\includegraphics[width = 8cm]{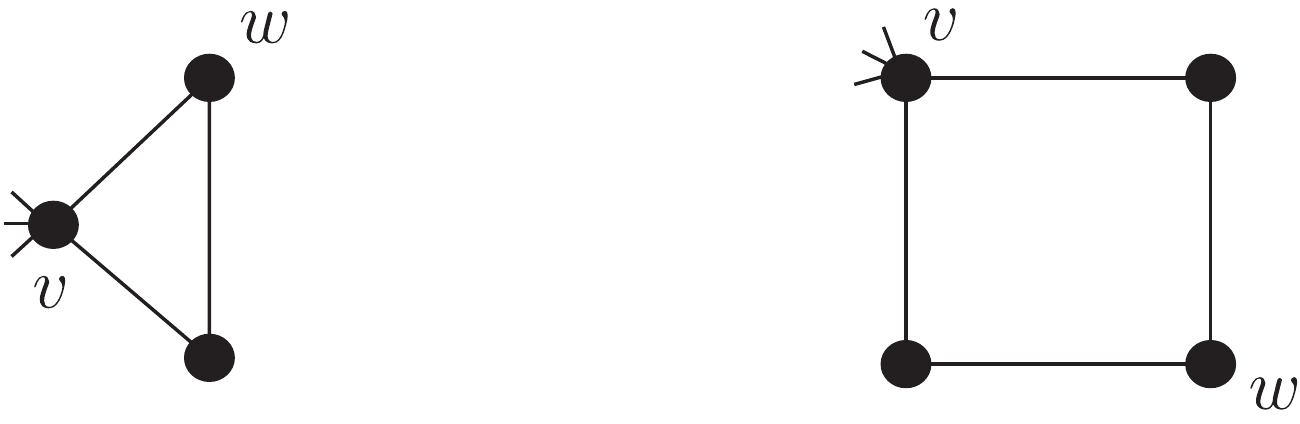}
\caption{\small{The different possibilities illustrating the proof of Proposition \ref{cactusprop}(2).}}
\label{cactusfig}
\end{figure}

If $G$ does not contain vertices of degree 1, it contains at least one cycle graph $C_m$ attached to the rest of the graph by just one vertex. If $m>4$, then Theorem \ref{csorba} leads to $I_G \sim_h \Susp I_H$, with $H$ being the graph obtained from $G$ after replacing $C_m$ by $C_{m-3}$, and by inductive hypothesis the result holds. Otherwise, either $m = 3$ or $m=4$, both possibilities shown in Figure \ref{cactusfig}. In both of them the vertex $v$ dominates $w$, hence Lemma \ref{dominationlemma} completes the proof.
\end{proof}

Motivated by a conjecture by Morton and Bae in \cite{Morton}, Manch\'on proved in \cite{pedro} that for any integer $n > 0$ there exists a bipartite planar graph whose independence number equals $n$ (and hence there are knots with arbitrary extreme coefficients in their Jones polynomial [see Subsection \ref{subseclando}]). In the following proposition we get the categorification of this example by computing its homotopy type, obtaining a wedge of as many spheres as desired.

\begin{proposition} \label{proppedro}
Let $G_r$ be the cactus graph consisting of a chain of hexagons disposed as shown in Figure \ref{pedrocat}(1). Then $I_{G_r} \sim_h \bigvee_{r+1 \mbox{\tiny{ copies}}} S^{2r-1}$.
\end{proposition}

\begin{figure}
\centering
\includegraphics[width = 10cm]{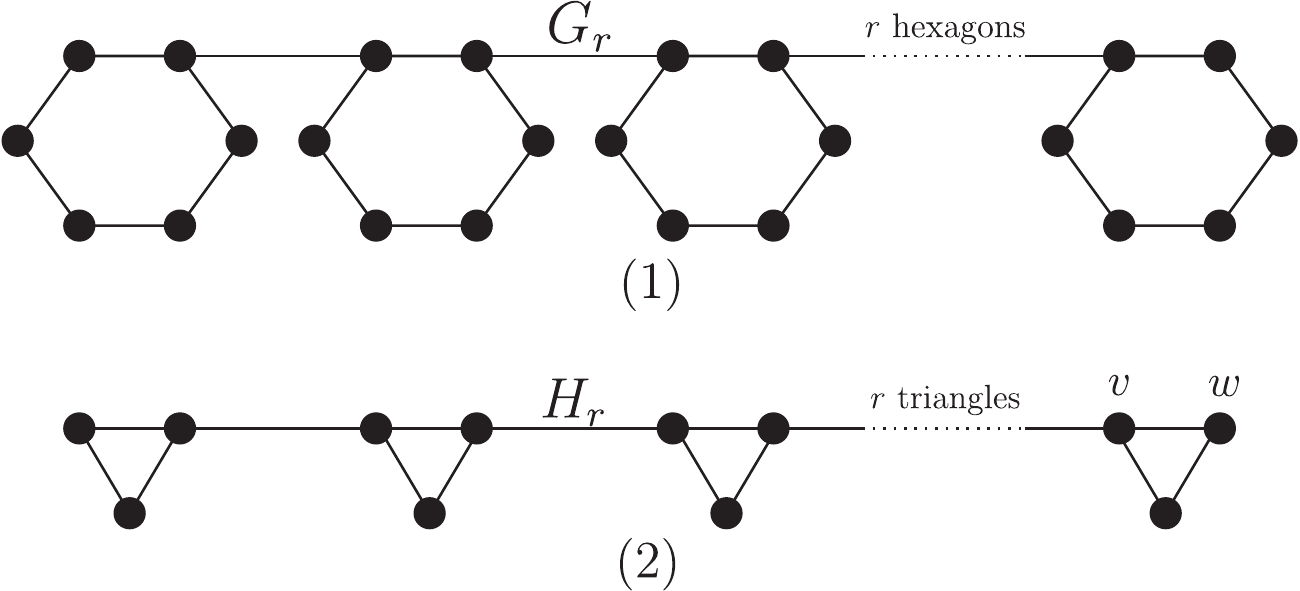}
\caption{\small{Graphs $G_r$ and $H_r$ illustrating Proposition \ref{proppedro}.}}
\label{pedrocat}
\end{figure}

\begin{proof}
By Theorem \ref{csorba} the proof can be reduced to study the independence complex of the graph $H_r$ depicted in Figure \ref{pedrocat}(2). As $v$ dominates $w$, applying Lemma \ref{dominationlemma} and Corollary \ref{addingleave} a finite number of times leads to the following homotopic equivalences
$$
I_{G_r} \sim_h \Susp^r I_{H_r} \sim_h \Susp^r \bigvee_{r+1 \mbox{ \tiny {copies}}} S^{r-1} \sim_h \bigvee_{r+1 \mbox{ \tiny {copies}}} S^{2r-1}.
$$
\end{proof}

\begin{example}\label{excactus}
Similarly to the example in Proposition \ref{proppedro} we present the homotopy type of the independence complexes of three families of cactus graphs consisting of different chains of triangles.

Let $A_r$ be the graph depicted in Figure \ref{ex1}, consisting of the wedge of $r$ triangles. Its independence complex $I_{A_r}$ is given by
$$I_{A_r} \sim_h S^{\lfloor\frac{r}{2}\rfloor} \vee S^{\lfloor\frac{r-1}{2}\rfloor} =
\left\{
  \begin{array}{lll}
     S^k \vee S^{k-1} & & \mbox{if } r = 2k, \\
     S^k \vee S^k & & \mbox{if } r = 2k+1.
  \end{array}
\right.
$$

\begin{figure}[h!]
\centering
\includegraphics[width = 12cm]{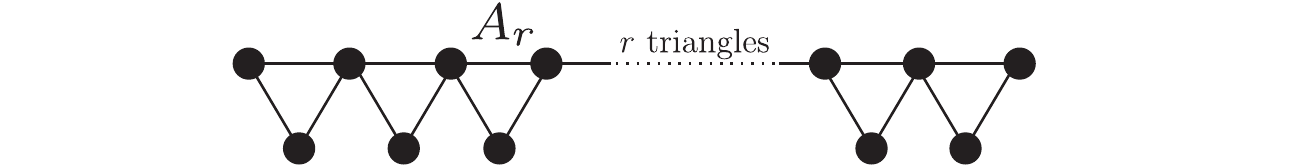}
\caption{\small{The graph $A_r$.}}
\label{ex1}
\end{figure}

Let $B_r$ and $C_r$ be the graphs depicted in Figure \ref{ex3}, consisting of the wedge of $r$ triangles and $r-1$ intervals (of length one or two alternatively). Their independence complexes have the same homotopy type. More precisely, if $f_r$ are the Fibonacci numbers, defined recursively as $f_1 = f_2 = 1$ and $f_n = f_{n-1} + f_{n-2}$, then

$$
I_{B_r} \sim_h I_{C_r} \sim_h \bigvee_{f_{r+2} \mbox{ \tiny {copies}}} S^{r-1}.
$$

\begin{figure}[h!]
\centering
\includegraphics[width = 11.5cm]{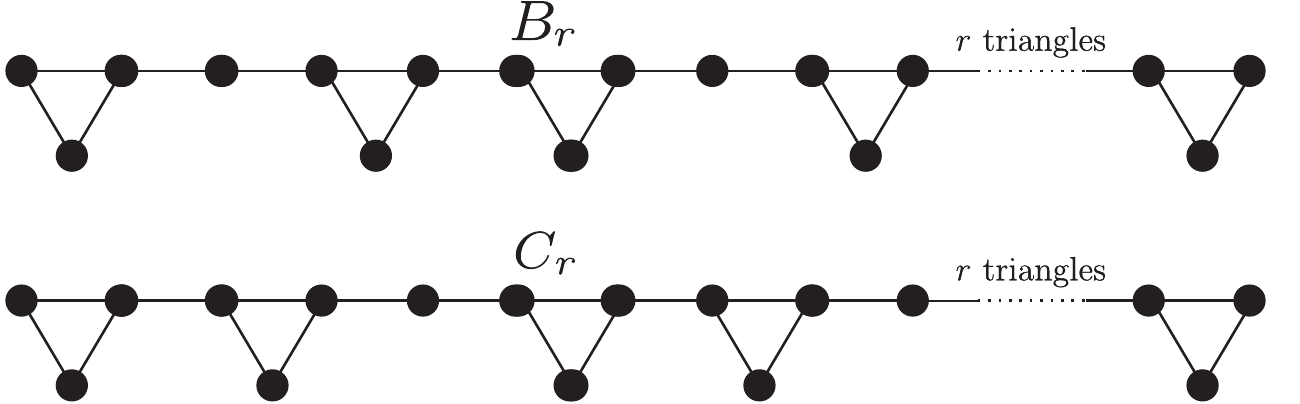}
\caption{\small{The graphs $B_r$ and $C_r$.}}
\label{ex3}
\end{figure}

\end{example}

\begin{proposition} \label{prop2}
Let $G$ be a non-trivial cactus graph whose cycles have order divisible by 3 (that is, its cycles are $3k$-gons) and containing at most one vertex of degree 1. Then $I_G$ is not contractible. Moreover, if $G$ does not contain any vertex of degree 1, then $I_G$ is a wedge of at least two spheres.
\end{proposition}

\begin{proof}
By Theorem \ref{csorba} is suffices to study the case with all cycles being triangles.

We proceed by induction on the number of edges of $G$. Note that if $G$ is not empty then it has at least 3 edges, as otherwise it would either be trivial or contain two leaves. As $I_{\emptyset} = S^{-1}$ and $I_{C_3} = S^0 \vee S^0$, the base cases hold. Now assume by inductive hypothesis that the statement holds when the graph contains less than $n$ edges, with $n>3$.

\begin{figure}
\centering
\includegraphics[width = 2.4cm]{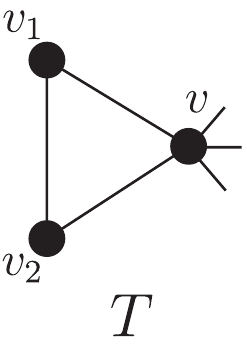}
\caption{\small{Vertices $v$ dominates both $v_1$ and $v_2$.}}
\label{cacprop2}
\end{figure}

Let $G$ be a graph satisfying the condition in the statement and containing $n$ edges. Suppose that it contains exactly one vertex $w$ of degree 1, and let $v$ be its associated preleaf. By Corollary \ref{addingleave} $I_G \sim_h \Susp I_{G-st(v)}$. Note that either $G-st(v)$ is empty or each of its connected components is either an interval $L_1$ (with $I_{L_1} = S^0$) or a cactus graph with less than $n$ edges. Therefore, the result holds after applying the inductive hypothesis and the fact that $I_{G_1 \sqcup G_2} \sim_h I_{G_1} * I_{G_2}$.

Otherwise, if $G$ does not contain any vertex of degree 1, then it contains at least one triangle $T$ with two vertices $v_1$ and $v_2$ having degree 2 (see Figure \ref{cacprop2}). Let $v$ be the other vertex of $T$, possibly with degree greater than 2. As $v$ dominates $v_1$, Lemma \ref{dominationlemma} leads to $I_G \sim_h I_{G-v} \vee \Susp I_{G-st(v)}$. A similar reasoning as before shows that neither $I_{G-v}$ nor $I_{G-st(v)}$ is contractible. This completes the proof.
\end{proof}

\begin{remark}
Note that Proposition \ref{prop2} cannot be extended to the case of cactus graphs having cycles of order not divisible by 3. Figure \ref{notdiv} shows examples with pentagons and squares whose associated independence complexes are contractible.

\begin{figure}
\centering
\includegraphics[width = 10.cm]{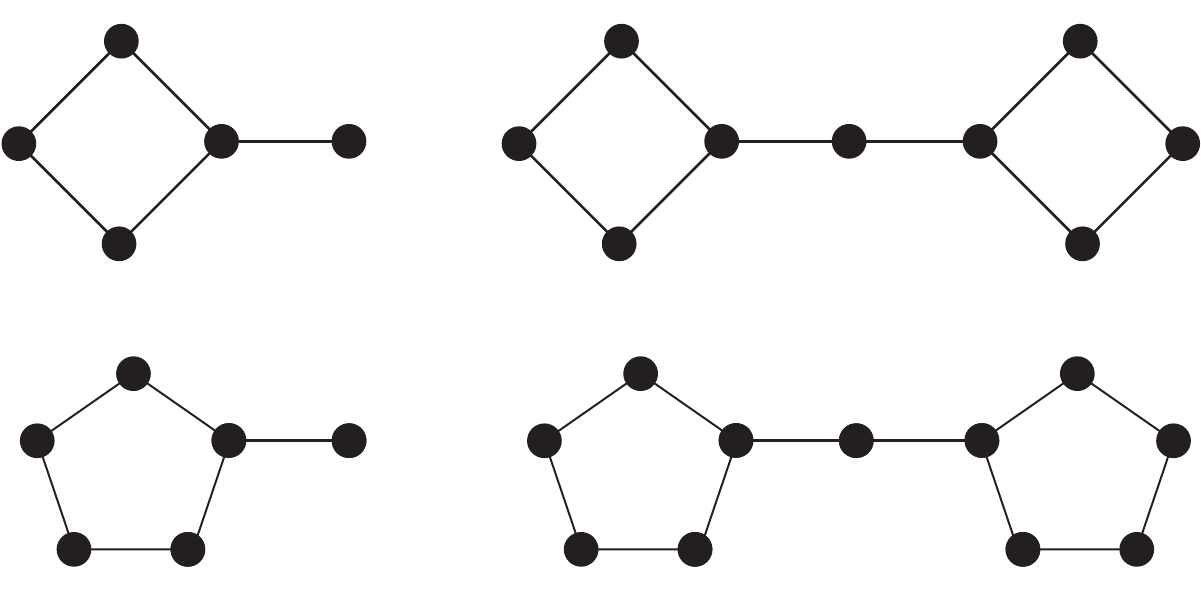}
\caption{\small{The independence complexes associated to these graphs are contractible.}}
\label{notdiv}
\end{figure}

\end{remark}

\subsection{Outerplanar graphs}\label{subouter}

\begin{definition} {\rm{\cite{outerplanar1}}} \label{outdef}
A simple connected graph $G$ is said to be outerplanar if it admits an embedding in the plane such that all vertices belong to the unbounded face of the embedding. A non-connected graph is outerplanar if all its connected components are so. An outerplane graph is a particular plane embedding of an outerplanar graph. Figure \ref{gluingedge} shows three outerplane graphs.

Outerplane graphs can be constructed from a single vertex by a finite number of the following operations: wedge with an interval $L_1$, wedge with a cycle graph and gluing a cycle graph along an edge of the unbounded region of the outerplane graph (the ``gluing along an outer edge'' operation, denoted by $|^1$, is illustrated in Figure \ref{gluingedge}).
\end{definition}

\begin{figure}
\centering
\includegraphics[width = 11.9cm]{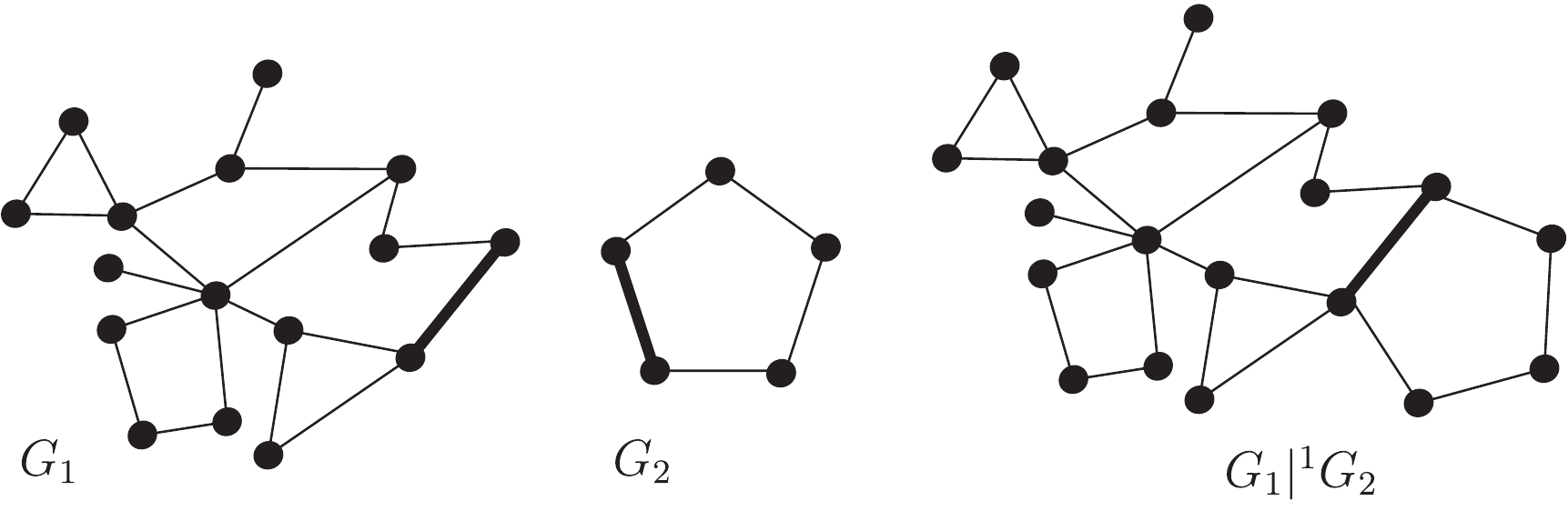}
\caption{\small{The outerplane graphs $G_1$, $G_2= C_5$ and $G_1 |^1 G_2$. The ``gluing edges'' have been thickened.}}
\label{gluingedge}
\end{figure}

Proposition \ref{cactusprop} has a natural generalization to the family of outerplanar graphs.

\begin{theorem} Let $G$ be an outerplanar graph. Then
\begin{enumerate}
\item [(1)] \cite{WP} $G$ is a circle graph.
\item[(2)] The independence complex of $G$, $I_G$, is homotopy equivalent to a wedge of spheres.
\end{enumerate}
\end{theorem}

\begin{proof} We use the standard inductive characterization of outerplanar graphs given in the second part of Definition \ref{outdef}.

(1) It is proven in \cite{WP}. Since the paper \cite{WP} is not easily available, we sketch a short inductive proof. We prove a slightly stronger statement, namely that for every outerplane graph $G$ there exists an associated chord diagram with the property that the two chords associated to each of the edges of the unbounded region of the graph have two of their endpoints close enough (that is, not separated by the endpoints of other chords; see Figure \ref{figouter1}).

Starting from the base case of a single vertex and assuming that the statement holds for outerplane graphs containing $n$ edges, we proceed by induction. The inductive step ``adding a cycle graph along an edge of the unbounded region'' follows from Lemma \ref{lemouter1}. The case of wedge operation holds by following the proof of Proposition \ref{cactusprop}(1) in a similar manner.

(2) The proof is analogous to the one of Proposition \ref{cactusprop}(2). The only additional thing one should check is the fact that gluing a cycle graph to an outerplanar graph along an outer edge preserves the property of its independence complex being homotopy equivalent to a wedge of spheres; this follows from Lemma \ref{lemouter2}.
\end{proof}

\begin{lemma}\label{lemouter2}
Let $G$ be a graph such that the independence complex $I_{G'}$ of any induced subgraph $G' \in G$ is homotopy equivalent to a wedge of spheres. Then $I_{G |^1 C_n}$ is homotopy equivalent to a wedge of spheres.
\end{lemma}

\begin{proof}
After applying a finite number of times Theorem \ref{csorba} it suffices to consider the graphs depicted in Figure \ref{figouter2}.

If $n = 3k+2$, then $I_{G |^1 C_n} \sim_h \Susp^k I_G$ and the result holds.

If $n = 3k$, then $I_{G |^1 C_n} \sim_h \Susp^{k-1} I_H$, with $H = G |^1 C_3$. As $v$ dominates $w$, by Lemma \ref{dominationlemma} $I_H \sim_h I_{H-v} \vee I_{H-st_H(v)}$. The proof follows from the hypothesis in the statement and Corollary \ref{addingleave}.

If $n = 3k+1$, then $I_{G |^1 C_n} \sim_h \Susp^{k-1} I_H$, with $H = G |^1 C_4$. Lemma \ref{dominationlemma} implies that $I_H \sim_h I_{H-v}$. Again the hypothesis in the statement together with Corollary \ref{addingleave} completes the proof.

\begin{figure}
\centering
\includegraphics[width = 10cm]{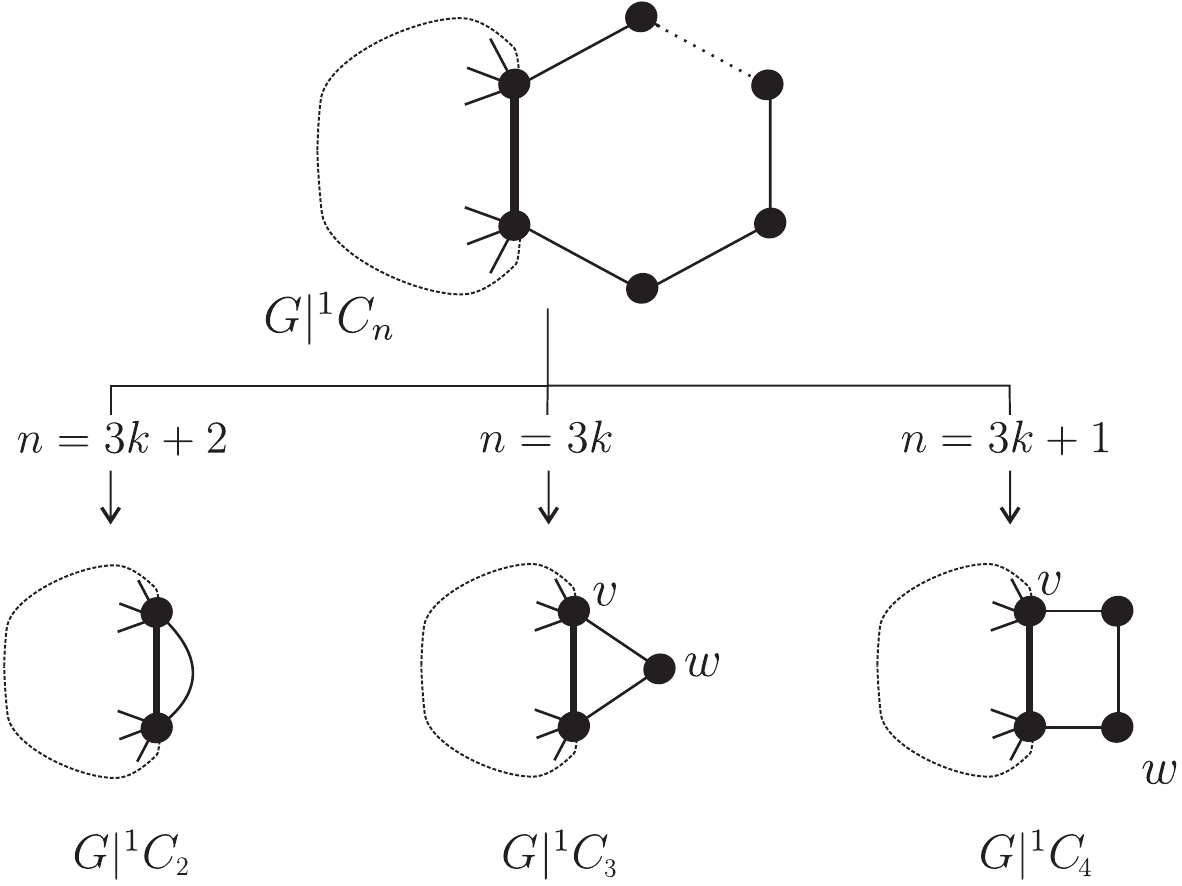}
\caption{\small{The graphs illustrating the proof of Lemma \ref{lemouter2}.}}
\label{figouter2}
\end{figure}

\end{proof}

\begin{lemma}\label{lemouter1}
Let $G$ be an outerplane graph with an associated chord diagram $\mathcal{C}_G$ such that every edge of the unbounded region of $G$ has associated two chords having two close endpoints (that is, with no chords between them). Then the graph $G |^1 C_n$ is an outerplane graph whose associated chord diagram $\mathcal{C}_{G|^1C_n}$ given in Figure \ref{figouter1} has the same property.
\end{lemma}

\begin{proof}
Figure \ref{figouter1} illustrates the proof.

\begin{figure}
\centering
\includegraphics[width = 11.5cm]{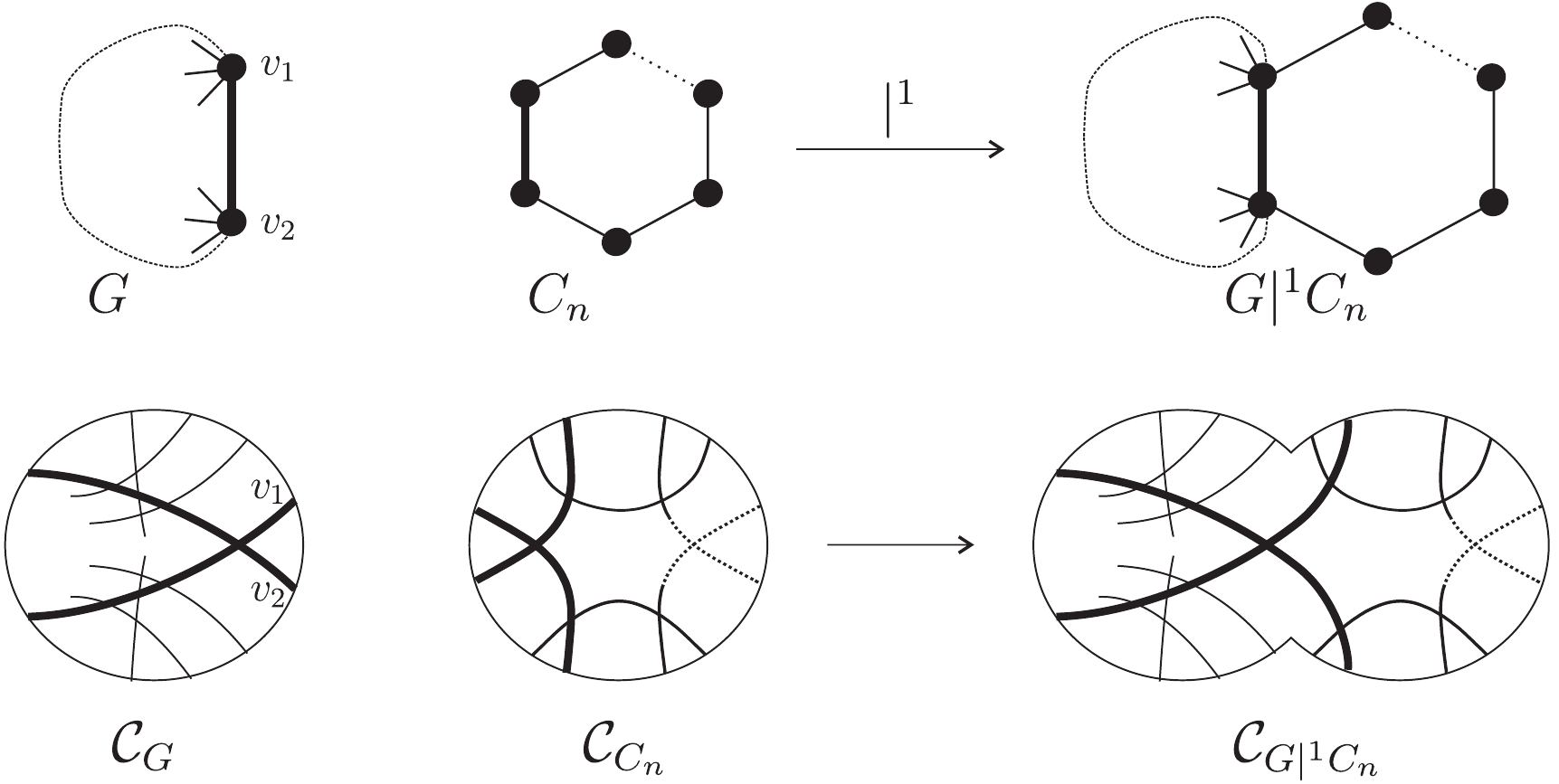}
\caption{\small{The chord diagrams $\mathcal{C}_G$, $\mathcal{C}_{C_n}$ and $\mathcal{C}_{G |^1 C_n}$ illustrating the proof of Lemma \ref{lemouter1} are shown.}}
\label{figouter1}
\end{figure}
\end{proof}

\section{Permutation graphs}\label{secpermut}

An interesting family of circle graphs are permutation graphs. We discuss them in this section.

\begin{definition}
A chord diagram $\mathcal{C}$ is said to be a permutation chord diagram if the boundary of $\mathcal{C}$ can be divided into two arcs $a_1$ and $a_2$ in such a way that each chord of $\mathcal{C}$ connects a point in $a_1$ with another one in $a_2$. The circle graph associated to a permutation chord diagram is called a \emph{permutation graph}. See Figure \ref{permuwed} for some examples.
\end{definition}

Permutation chord diagrams on $n$ chords are in one-to-one correspondence with the permutation group of $n$ elements $S_n$.

As suggested by Micha{\l} Adamaszek, in the following result we prove that Conjecture \ref{conj} holds for permutation graphs.

\begin{theorem}
Let $G$ be a permutation graph. Then, its independence simplicial complex $I_G$ is homotopy equivalent to a wedge of spheres.
\end{theorem}

\begin{proof}
We proceed by induction on the number of chords in the associated chord diagram $\mathcal{C}$ (that is, on the number of vertices in $G$, as chords and vertices are in one-to-one correspondence). The base cases when either there are no chords or there is just one hold, as $I_{\emptyset} = S^{-1}$ and $I_{\{v\}}$ is contractible. Suppose that the statement holds for chord diagrams with at most $n-1$ chords, $n\geq2$. Let $G$ be the permutation graph arising from a permutation chord diagram with $n$ chords. If $G$ is not connected, then $G = G_1 \sqcup G_2$ and $I_G = I_{G_1} * I_{G_2}$, hence by the inductive hypothesis $I_G$ is homotopy equivalent to a wedge of spheres. Suppose now that $G$ is connected. Starting from the leftmost upper side of the circle, number the chords as $\{s_1, s_2 \ldots, s_n\}$ as shown in Figure \ref{permuorder}. Let $v_i$ be the vertex in $G$ associated to the strand $s_i$, $1 \leq i \leq n$. Then, according to the relative positions between $s_1$ and $s_2$ there are two cases (again, see Figure \ref{permuorder}):

\begin{figure}
\centering
\includegraphics[width = 9cm]{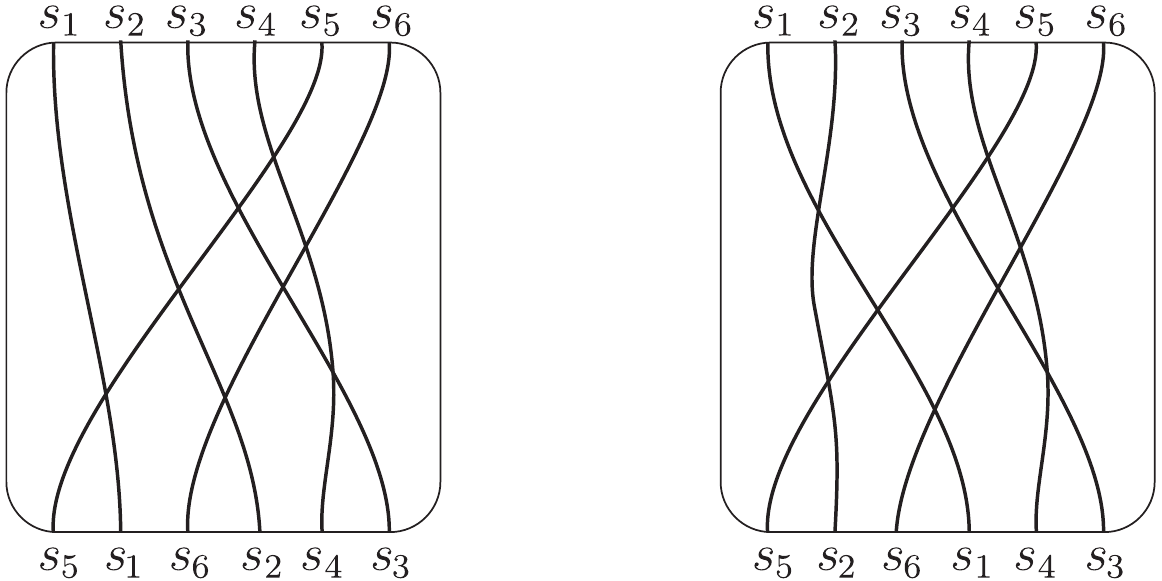}
\caption{\small{The two possible cases of domination according to the relative position of strands $s_1$ and $s_2$.}}
\label{permuorder}
\end{figure}

If $s_1$ and $s_2$ do not intersect, then $v_2$ dominates $v_1$. Hence, by Lemma~\ref{dominationlemma}(1) $I_G \sim_h I_{G-v_2}$. Otherwise $v_1$ dominates $v_2$, and again by Lemma~\ref{dominationlemma}(2) $I_G \sim_h I_{G-v_1} \vee \Susp I_{G-st(v_1)}$. Applying the inductive hypothesis completes the proof in both cases.
\end{proof}

The following result is based on the converse idea, as it shows that any wedge of spheres $\mathcal{X}$ can be realized as the independence complex of a permutation graph, assuming that if $\mathcal{X}$ has $m$ components, then $m-1$ of them are isolated points:

\begin{proposition}\label{anywedge}
For any wedge of spheres $\mathcal{X} = S^{k_1} \vee \ldots \vee S^{k_n}$ there exists a permutation graph $G$ whose independence complex $I_G$ is homotopy equivalent to $\mathcal{X}$.
\end{proposition}

\begin{proof}
We present a constructive proof by showing an algorithm for constructing the permutation chord diagram $\mathcal{C}_G$ leading to the graph $G$ in the statement. We will use two different ``moves'':

Move I: adding a chord $s_0$ (from upper-leftmost to bottom-rightmost sides) crossing all the previous chords in $\mathcal{C}_G$. The effect of $s_0$ in $G$ is adding a new vertex $v_0$ connected with all the previous vertices; hence, $v_0$ is isolated in $I_G$. 

Move II: adding (at the rightmost part of $\mathcal{C}_G$) two chords crossing each other and not intersecting the other chords. This implies adding two vertices joined by an edge in $G$, which is equivalent to taking suspension in its independence complex.

Example \ref{expermwed} illustrates the method for constructing the permutation chord diagram. Order the wedges of spheres in $\mathcal{X}$ so their dimensions decrease, and rewrite $\mathcal{X}$ in terms of suspensions and wedges of $S^0$ by using the fact that $S^m \vee S^n \sim_h \Susp(S^{m-1}\vee S^{n-1}).$

The permutation chord diagram $\mathcal{C}_G$ is constructed from the innermost level to the outermost one (the nesting-level depends on the number of suspensions acting over it). In the innermost level one finds either $S^0$ or wedges of $S^0$ (corresponding to the original spheres with the highest dimension); for each wedge of $S^0$, perform move I. Once a ``nesting level'' is completed, consider its associated suspension (move II) and move to the previous level. As $k_i$ are finite, this process finishes after repeating this procedure a finite number of times.
\end{proof}

\begin{example}\label{expermwed}
We show how to find a permutation chord diagram $\mathcal{C}_G$ whose associated permutation graph $G$ has an independence complex $I_G$ homotopy equivalent to $S^3 \vee S^2 \vee S^1$. \\

Figure \ref{permuwed} illustrates the following steps: \\
\begin{figure}
\centering
\includegraphics[width = 11cm]{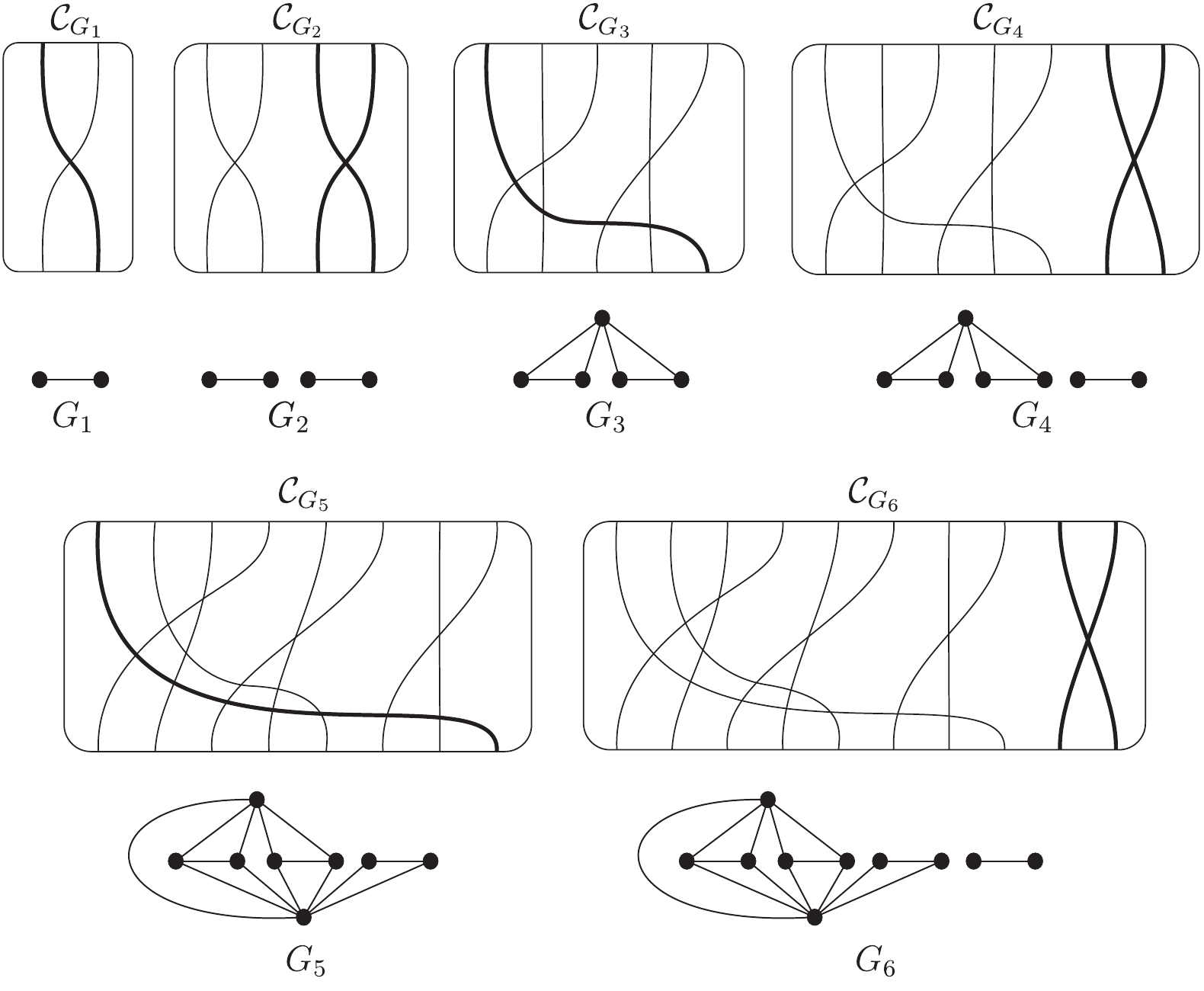}
\caption{\small{These chord diagrams illustrate the six steps in Example \ref{expermwed}. The new chords in each step have been thickened.}}
\label{permuwed}
\end{figure}

\noindent 0.- Decompose $S^3 \vee S^2 \vee S^1$ as $\Susp (\Susp (\Susp (S^0) \vee S^0) \vee S^0)$. \\
1.- Move I: $I_{G_1} \sim_h S^0$. \\
2.- Move II: $I_{G_2} \sim_h \Susp (S^0) \sim_h S^1$. \\
3.- Move I: $I_{G_3} \sim_h \Susp (S^0) \vee S^0 \sim_h S^1 \vee S^0$. \\
4.- Move II: $I_{G_4} \sim_h \Susp(\Susp (S^0) \vee S^0 \sim_h S^2 \vee S^1$. \\
5.- Move I: $I_{G_5} \sim_h \Susp(\Susp (S^0) \vee S^0) \vee S^0 \sim_h S^2 \vee S^1 \vee S^0$. \\
6.- Move II $I_{G_6} \sim_h \Susp(\Susp(\Susp (S^0) \vee S^0) \vee S^0) \sim_h S^3 \vee S^2 \vee S^1$. \\
\end{example}

Permutation graphs are usually not bipartite. However in some interesting cases we can change them into bipartite circle (Lando) graphs, as explained in Remark \ref{remcsorba}. In Examples \ref{gap1} and \ref{gap2} we construct Lando graphs with $I_G$ homotopy equivalent to $S^{n+k}\vee S^{2n-1+k}$, for any $n \geq 1$ and $k \geq 0$, and  $S^{2m+2n+k} \vee S^{m+2n+1+k} \vee S^{m+n+1+k}$, for non-negative integers $m,n,k$.

\begin{remark}\label{remcsorba}
Note that by replacing some edges by paths of length four any graph can be turned into a bipartite graph. Sometimes the property of a graph being a circle graph is preserved by this transformation. In Figure \ref{efectocsorba} we show this transformation at the level of both a graph and its associated chord diagram, in the particular case when two chords intersect close enough to the boundary of the disc (meaning that there are not other chords between two of their endpoints). This construction is used several times throughout the paper.
\end{remark}

\begin{figure}
\centering
\includegraphics[width = 8.5cm]{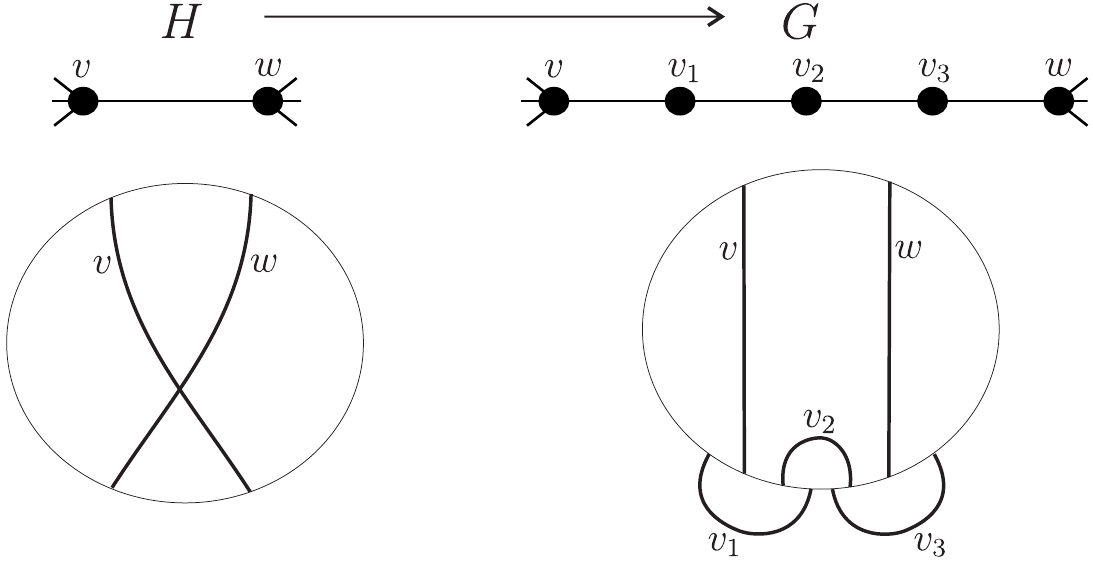}
\caption{\small{The effect of replacing an edge by a path of length four in a circle graph and its effect at the level of chord diagrams, under the conditions described in Remark \ref{remcsorba}.}}
\label{efectocsorba}
\end{figure}

\begin{example} \label{gap1}
Figure \ref{figgap1} shows a family of chord diagrams $\mathcal{C}_n^k$, with $n>1$ and $k\geq0$, whose associated bipartite circle graphs $G_n^k$ satisfies $$I_{G_n^k} \sim_h S^{n+k}\vee S^{2n-1+k}.$$

Starting from $\mathcal{C}_n$, whose associated Lando graph $G_n$ consists of the wedge of $n$ triangles (with $I_{G_n} \sim_h S^0 \vee S^{n-1}$), $\mathcal{C}_n^k$ is obtained after performing $n$ times the transformation in Theorem \ref{csorba} and illustrated in Figure \ref{efectocsorba}, and adding $k$ pairs of chords disposed as in Figure \ref{figgap1}, leading to $n+k$ suspensions of $I_{G_n}$ at the level of independence complexes.

\begin{figure}
\centering
\includegraphics[width = 10.5cm]{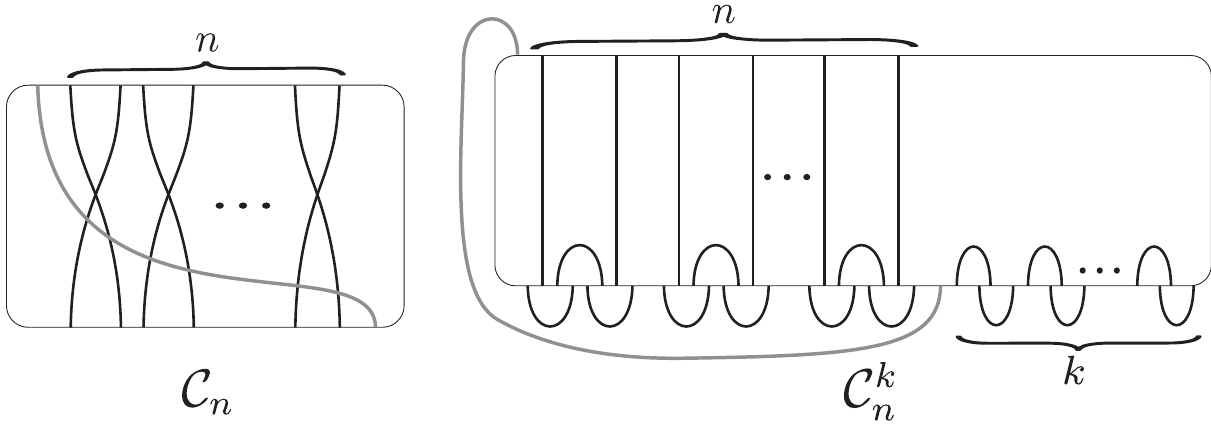}
\caption{\small{The chord diagrams illustrating Example \ref{gap1}.}}
\label{figgap1}
\end{figure}
\end{example}

\begin{example} \label{gap2}
For given non-negative integers $m,n,k$, define $G_{m,n}^k$ as the bipartite circle graph associated to the chord diagram $\mathcal{C}_{m,n}^k$ depicted in Figure \ref{figgap2}. Then
$$I_{G_{m,n}^k} \sim_h S^{2m+2n+k} \vee S^{m+2n+k+1} \vee S^{m+n+k+1}.$$

The construction is similar to that in Example \ref{gap1}. We start from $\mathcal{C}_{m,n}$, the chord diagram shown in Figure \ref{gap2}, whose associated circle graph $G_{m,n}$ consists of a vertex connected with all the vertices contained on the disjoint union of a wedge of $m$ triangles and $n$ intervals. Then we apply the transformation in Figure \ref{efectocsorba} to the $m+n+1$ pairs of chords in the diagram. As $I_{G_{m,n}} \sim_h S^0 \vee S^n \vee S^{m+n-1}$, the result holds after applying Theorem \ref{csorba} and adding $k$ pair of chords as those shown in Figure \ref{figgap2}.

\begin{figure}
\centering
\includegraphics[width = 10.5cm]{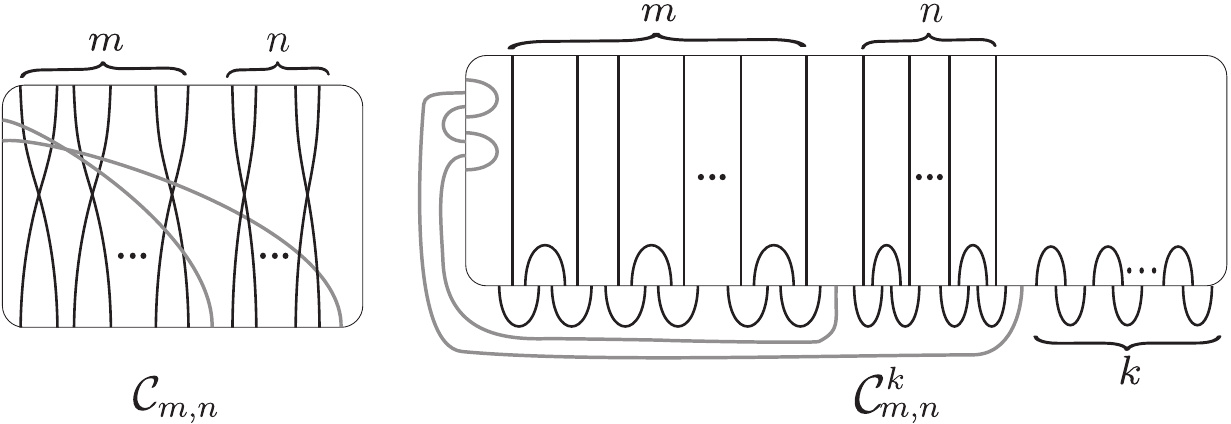}
\caption{\small{The chord diagrams illustrating Example \ref{gap2}.}}
\label{figgap2}
\end{figure}
\end{example}

In Subsection \ref{hthick} we explore links related to bipartite circle (Lando) graphs described in these examples.

\section{Non-nested circle graphs} \label{secnested}

In this section we prove Conjecture \ref{conj} for the family of \emph{non-nested circle graphs}, which are bipartite. The study of bipartite circle graphs is relevant for us, since these are the graphs arising as Lando graphs associated to link diagrams (see Section \ref{sec7}).

\begin{definition}\label{defnonnested}
A bipartite chord diagram $\mathcal{C}$ is said to be a \emph{non-nested chord diagram} if either the inner or outer region bounded by the associated circle does not contain nested chords.  A \emph{non-nested circle graph} is a graph arising from a non-nested chord diagram.
\end{definition}

\begin{theorem}
Let $G$ be a non-nested circle graph. Then its independence simplicial complex $I_G$ is homotopy equivalent to a wedge of spheres.
\end{theorem}

\begin{proof}
We proceed by induction on the number of vertices of $G$. It is clear that the statement holds for an empty graph, as well as a graph with
one or two vertices. Now, assume as inductive hypothesis that it holds for any non-nested graph with $n$ vertices, and assume that $G$ is a non-nested graph containing $n+1$ vertices, $n \geq 2$. The inductive step falls within one of the following possibilities:
\begin{enumerate}
%\item[(1)] If $G$ is not connected, then $G = G_1\sqcup G_2$ with $|V(G_1)|,|V(G_2)| \leq n$. By inductive hypothesis $I_{G_1}$ and $I_{G_2}$ are
%homotopy equivalent to wedges of spheres. Since $I_{G_1\sqcup G_2}= I_{G_1}*I_{G_2}$, $I_G$ is homotopy equivalent to a wedge of spheres.

\item[(1)] If $G$ contains a leaf, with associated preleaf $v$, then $I_G\sim \Sigma I_{G-st(v)}$ by Corollary \ref{addingleave}, and by inductive hypotesis $I_{G-st(v)}$ is homotopy equivalent to a wedge of spheres, so is $I_G$.

\begin{figure}
\centering
\includegraphics[width = 9cm]{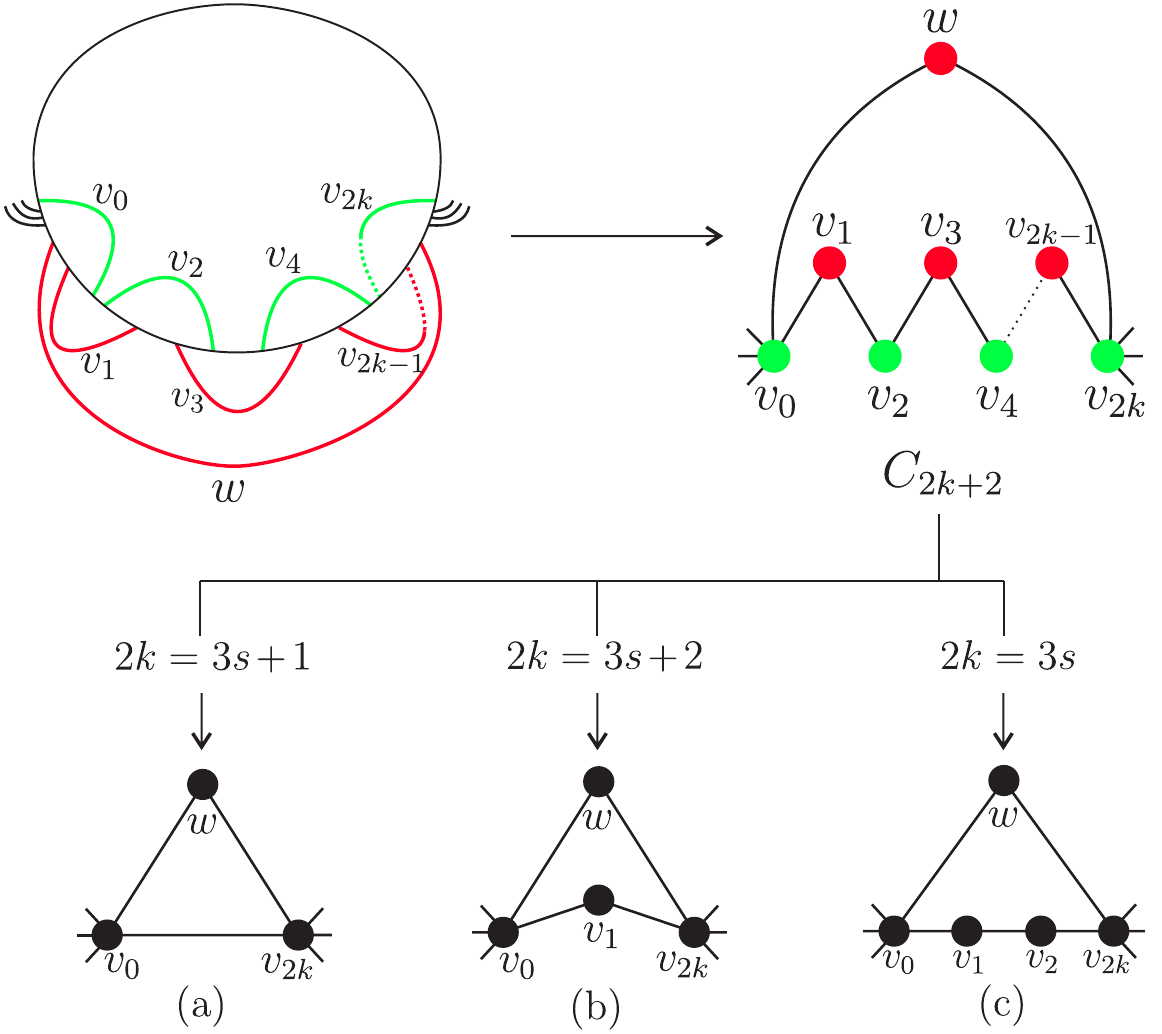}
\caption{\small{The chord diagram leading to the graph $C_{2k+2}$ and its transformations after applying $s$, $s$ and $s-1$ times Corollary~\ref{csorba}, respectively. In the first row the case $k=3$ is shown when considering the dotted lines as if they were solid.}}
\label{fignonnest}
\end{figure}

\item[(2)] Assume that $G$ has no vertices of degree $1$; hence, it contains an inner (in the sense that it bounds a region with no vertices) even-cycle graph, $C_{2k+2}$, as shown in Figure \ref{fignonnest}. Let $L_{2k}$ be the subgraph of $C_{2k+2}$ consisting of the $v_0, v_1, \ldots, v_{2k}$ path of length $2k$. Then, by applying a finite number of times Theorem \ref{csorba}, the length of $L_{2k}$ can be reduced to 1, 2 or 3, depending on the value of $2k$ module 3 (note that the new graphs are not necessarily bipartite). Figure \ref{fignonnest} illustrates the different cases:
    
    (a) If $2k = 3s + 1$, then $I_G \sim_h \Sigma^s I_{G-L_{2k}-v_{2k}} \vee \Sigma^{s+1} I_{G-L_{2k}-v_0-st_G(v_{2k})}$, where $G - L_{2k}$ is the graph obtained by removing the $2k-1$ inner vertices of the path $L_{2k}$. After applying $s$ times Corollary~\ref{csorba}, one gets $I_G \sim_h \Sigma^s I_H$, with $H$ being the graph obtained from $G$ by contracting $C_{2k+2}$ to $C_3$ according to Figure \ref{fignonnest}(a). Now $v_{2k}$ dominates $w$ in $H$ and therefore by Lemma \ref{dominationlemma}(2) $I_H \sim_h I_{H-v_{2k}} \vee \Sigma I_{H-st_H(v_{2k})}$. Hence $I_G \sim_h \Sigma^s I_{G-L_{2k}-v_{2k}} \vee \Sigma^{s+1} I_{G-L_{2k}-v_0-st_G(v_{2k})}$.

    (b) If $2k = 3s + 2$, then $I_G \sim_h \Sigma^s I_{G - L_{2k}}$. The result holds after applying $s$ times Corollary \ref{csorba} and Lemma \ref{dominationlemma}(1), since $v_1$ dominates $w$ in Figure \ref{fignonnest}(b).

    (c) If $2k = 3s$, then $I_G \sim_h \Sigma^{s+1} I_{G - L_{2k} - st_G(v_0) - st_G(v_{2k})}$. To get this result start by applying $s-1$ times Corollary \ref{csorba} to get $I_G \sim_h \Sigma^{s-1} I_H$, with $H$ being the graph obtained from $G$ by contracting $C_{2k+2}$ to $C_5$ according to Figure \ref{fignonnest}(c). By Corollary \ref{pathwithL3} $I_{H-v_{2k}}$ is contractible and therefore by Proposition \ref{Reiner} $I_H \sim_h \Sigma I_{H-st_H(v_{2k})}$. Corollary \ref{addingleave} implies $I_H \sim_h \Sigma^2 I_{H-st_H(v_{2k})-st_H(v_0)}$. Therefore, $I_G \sim_h \Sigma^{s+1} I_{G - L_{2k} - st_G(v_0) - st_G(v_{2k})}$.

Since the four graphs appearing in the previous expressions of $I_G$ are non-nested graphs, the inductive hypothesis completes the proof.
\end{enumerate}
\end{proof}

\begin{remark}
Starting from an outerplanar graph $G$, consider its barycentric subdivision $G_2$. The bipartite graph $G_2$ is easily recognized to be a non-nested circle graph (see Figure \ref{baricentric}). This observation is outside the scope of this paper, but it is related to the statement by Csorba ``$I_{G_2}$ is homotopy equivalent to the suspension of the combinatorial Alexander dual of $I_G$'' and by Cabello and Jejcic ``$G$ is an outerplanar graph if and only if $G_2$ is a circle graph'' in \cite[Theorem 6]{Csorba} and \cite{Cabello}, respectively. These results deserve attention in the study of the independence complexes of circle graphs.
\end{remark}

\begin{figure}
\centering
\includegraphics[width = 11.2cm]{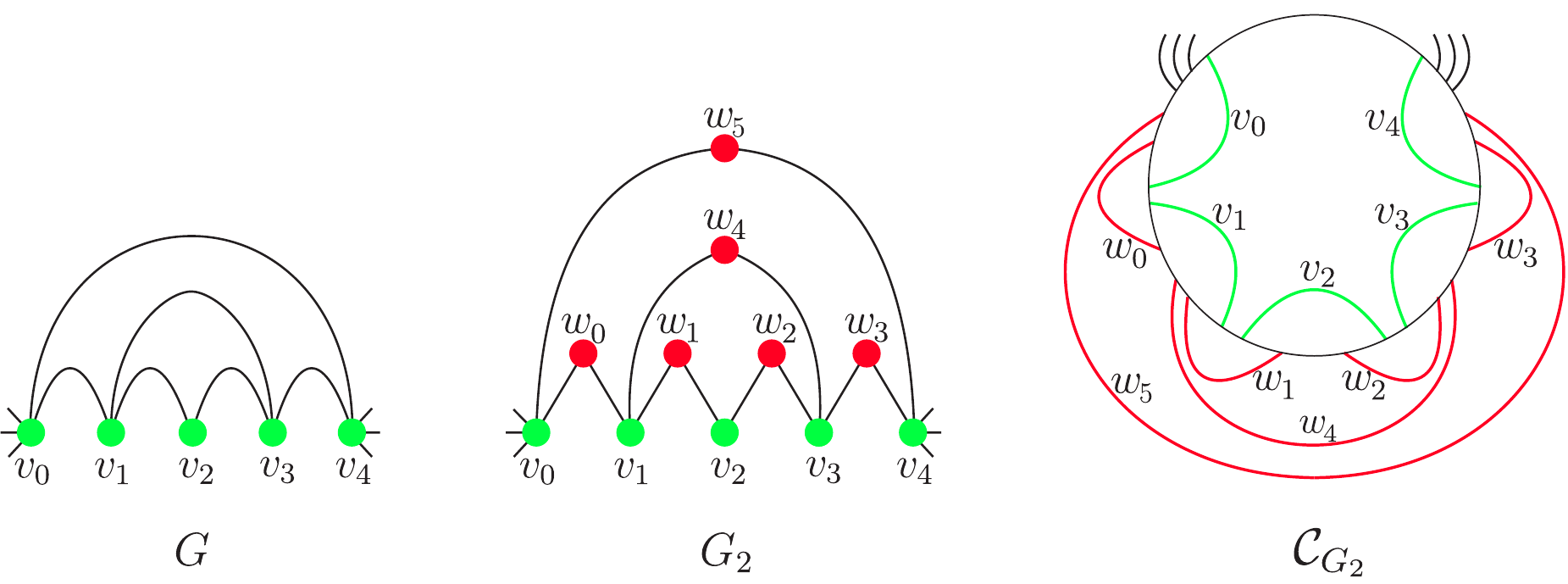}
\caption{\small{An outerplanar graph $G$, its barycentric subdivision $G_2$ and its associated non-nested chord diagram $\mathcal{C}_{G_2}$.}}
\label{baricentric}
\end{figure}

\section{Structure Theorem}\label{sec6}

In this section we apply the ``principle for gluing homotopies" \cite{Bro-1, Brown, Bjo} to obtain several useful properties of graphs and their
independence complexes. In particular, Theorem \ref{globaltheo} generalizes Theorem \ref{csorba} by Csorba \cite{Csorba}, the bipartite suspension theorem by Nagel and Reiner \cite{Nagel} and its generalization by Barmak and Jonsson \cite{Barmak, Jonsson}.

We start from a series of simple but useful lemmas. Given a vertex $v$ of a loopless graph $G$, write $A_v = I_{G-st(v)} * v$.

\begin{lemma} \label{Lema1}
\begin{enumerate}
\item[(1)] Let $G$ be a loopless graph and $v$ a vertex in $G$, with $lk(v) = \{v_0, v_1 \ldots, v_k \}$. Then
$$I_G = \bigcup_{i=0}^k A_{v_i} \cup A_v.$$

\item[(2)] Let $s = \{v_0, \ldots, v_n\}$ be an independent set in a graph $G$. Thus they constitute an $n$-dimensional simplex in $I_G$, $\Delta^n$. Then
$$\bigcap_{i=0}^{n} A_{v_i} = I_{G - st(v_0) - \ldots - st(v_n)} * \Delta^n,$$
which is contractible.
\end{enumerate}
\end{lemma}

\begin{proof}(1) If $s \in I_G$, then $s$ either contains one of the vertices in the link of $v$, or $s \cup \{v\}$ is an independent set.

(2) The left side of the equality in (2) describes the flag simplicial subcomplex, $\mathcal{X}$, of $I_G$ whose simplices are characterized by the property that if $\sigma \in P(\mathcal{X})$ then the simplex $\sigma \cup v_i$ is also in $P(\mathcal{X})$ for any $0 \leq i \leq n$.  Let $\sigma$ be a simplex in $P(\mathcal{X})$. Since $s$ and $\sigma$ are independent sets in $G$ and $\sigma \cup v_i$ is an independent set for every $i$, then $\sigma \cup s$ is an independent set by the flag property of $\mathcal{X}$, hence a simplex in $P(\mathcal{X})$. Thus $\mathcal{X} = I_{G-st(v_o)-...-st(v_n)} * \Delta^n$.
\end{proof}

\begin{theorem} {\rm{\cite{Brown}}} \label{brown}
Let $\mathcal{X}$ be a simplicial complex, $A$ and $B$ two subcomplexes such that $A \cup B = \mathcal{X}$. If $A$ and $B$ are contractible then $\mathcal{X}$ is homotopy equivalent to the suspension of the intersection of $A$ and $B$, that is, $\mathcal{X}  \sim_h \Sigma (A \cap B)$. In particular if $A \cap B$ is contractible then $\mathcal{X}$ is contractible.
\end{theorem}

%The next ``glueing lemma" can be concluded from the Borsuk result \cite{Bor} that if $\mathcal{K}$, $\mathcal{K}_1$, and $\mathcal{K}_2$ are simplicial complexes such that $\mathcal{K}=\mathcal{K}_1\cup \mathcal{K}_2$ are simplicial complexes and $\mathcal{K}_1$, $\mathcal{K}_2$ and $\mathcal{K}_1 \cap \mathcal{K}_2$ are contractible then $\mathcal{K}$ is contractible (see also Theorem \ref{brown}).

\begin{lemma} \label{Lema2}
Let $G$ be a loopless graph and $s = \{v_1,v_2, \ldots, v_n\}$ an independent set. Then the simplicial complex $(A_1\cup \ldots \cup A_k)\cap A_{k+1} \cap \ldots \cap A_n$ is contractible for every $1 \leq k \leq n$. In particular, $A_1\cup A_2\cup \ldots \cup A_n$ is contractible.
\end{lemma}

\begin{proof}
We proceed by induction on pairs $(n,k)$ ordered lexicographically, that is $(n_1,k_1) < (n_2,k_2)$ if either $n_1 < n_2$ or $n_1 = n_2$ and $k_1 < k_2$. By Lemma \ref{Lema1} the product $A_1\cap A_2\cap \ldots \cap A_n$ is contractible, hence the base case when $k=1$ (any value of $n$) holds. Now, in the inductive step, consider $(n,k)$ with $k > 1$ and assume that the result holds for smaller pairs. We have
$$(A_1 \cup \ldots \cup A_k) \cap A_{k+1} \cap \ldots \cap A_n = $$ $$((A_1\cup \ldots \cup A_{k-1}) \cap (A_{k+1}\cap \ldots \cap A_n)) \cup (A_k \cap (A_{k+1}\cap \ldots \cap A_n)).$$

The summands (with respect to $\cup$) in the right side of the equality are contractible by inductive assumption. Next, by Theorem \ref{brown}, to prove that the whole complex is contractible it suffices to check that their product ($\cap$) is contractible. We get
$$((A_1 \cup \ldots \cup A_{k-1})\cap (A_{k+1} \cap \ldots \cap A_n)) \cap (A_k\cap (A_{k+1} \cap \ldots \cap A_n)) = $$
$$(A_1 \cup \ldots \cup A_{k-1})\cap A_k \cap A_{k+1} \cap \ldots \cap A_n,$$
which is contractible by the inductive assumption.
\end{proof}

Now we present our main result in this section, which generalizes and extends several results, as will be shown in the subsequent corollaries.

\begin{theorem}\label{globaltheo}
Let $G$ be a loopless graph, $v \in G$ a vertex of degree $n$, such that the set $lk(v)=\{v_1,...,v_n\}$ is an independent set. Write $W_i = lk(v_i) - v$ ($1 \leq i \leq n$) and define $\mathcal{K}(G,v)$ as the complex obtained from $I_{G-st(v)}$ by deleting the simplexes $s$ such that $s \cap W_i \neq \emptyset$ in $I_{G-st(v)}$ for every i. Then $I_G$ is homotopy equivalent to the suspension of $\mathcal{K}(G,v)$.
\end{theorem}

\begin{proof}
By definition $A_v$ and $A_{v_i}$ are contractible, for every $1 \leq i \leq n$. Moreover, as the vertices in $lk(v)$ are independent, by Lemma \ref{Lema2} $\cup_{i=1}^k A_i$ is contractible. Thus by Lemma \ref{Lema1} $I_G = \cup_{i=1}^k A_i \cup A_v$, and from Theorem \ref{brown} it holds that $I_G$ is homotopy equivalent to the suspension of $(\cup_{i=1}^k A_i) \cap A_v$.

It remains to identify this simplicial complex with $\mathcal{K}(G,v)$. Firstly, note that any simplex $s \in (\cup_{i=1}^k A_i) \cap A_v$ contains neither $v$ nor $v_i$ ($1 \leq i \leq n$) as a vertex. Furthermore, $s \cap W_i$ must be empty for some $i$ (and this is the reason for the deletion in the statement). There are not other restrictions for $s$, thus the simplexes of $(\cup_{i=1}^k A_i) \cap A_v$ are exactly the same as those in $\mathcal{K}(G,v)$.
\end{proof}

\begin{remark}
Note that in the statement of Theorem \ref{globaltheo} we do not assume that $W_i \cap W_j = \emptyset$. Moreover, in general $\mathcal{K}(G,v)$ is not a flag simplicial complex when $n \geq 3$, as illustrated by Examples \ref{nonext1} and \ref{nonext2}.
\end{remark}

\begin{example}\label{nonext1}
Let $G_n$ be the star of $n$ rays of length 2, that is, the wedge of $n$ copies of $L_2$ keeping a vertex $v$ as basepoint (see Figure \ref{star}). For $1 \leq i \leq n$, let $w_i$ be the leaves of $G_n$ and $v_i$ their associated preleaves. Apply now Theorem \ref{globaltheo} taking the wedge vertex as $v$, and $W_i = \{w_i\}$, for every $i$. Hence \, $I_{G_n-st(v)} = \Delta^{n-1}$ and $\mathcal{K}(G_n,v)$ is constructed from it by deleting the simplices having non-empty intersection with every $w_i$, that is, deleting the geometric interior of $\Delta^{n-1}$. Hence $\mathcal{K}(G_n,v)$ is homotopy equivalent to $S^{n-2}$ and therefore, taking its suspension, $I_{G_n} = S^{n-1}$. We confirm this by using Corollary \ref{addingleave} over $w_1$, leading to the sequence $I_{G_n} \sim_h \Susp I_{G_n-st(v_1)} = \Susp I_{\underbrace{(L_1 * \ldots * L_1)}_\text{n-2}} \sim_h S^{n-1}$, as expected.

\begin{figure}
\centering
\includegraphics[width = 3.5cm]{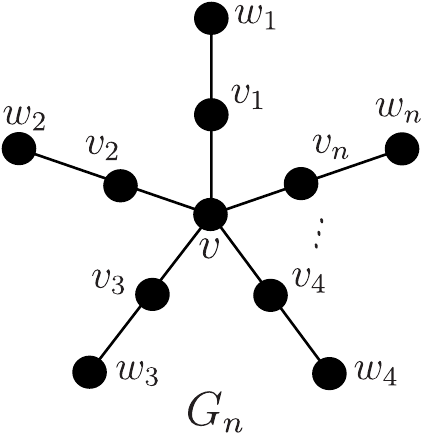}
\caption{\small{$G_n$ is the star of $n$ rays of length 2.}}
\label{star}
\end{figure}
\end{example}

\begin{example}\label{nonext2}
Let $G$ be the wedge of two hexagons. See Figure \ref{figglobaltheo}. We apply Theorem \ref{globaltheo} taking the wedge vertex as $v$, and we get that $lk(v) = \{v_1, v_2, v_3, v_4\}$, and $I_{G-st(v)} = \{\emptyset, 1, 2, 3, 4, 5, 6, 13, 14, 15, 16, 24, 25, 26, 34, 35, \\ 36, 46, 134, 135, 136, 146, 246, 346, 1346\}$. Hence $\mathcal{K}(G,v)$ is constructed from $I_{G-st(v)}$ by deleting the simplex $1346$, which is the only one intersecting $W_i$, for every $1 \leq i \leq 4$. Therefore, $\mathcal{K}(G,v)$ is homotopy equivalent to $S^1 \vee S^2$, and taking its suspension we get that $I_G \sim_h S^2 \vee S^3$, as expected from the computations in Example \ref{exwedge}.
\end{example}

\begin{figure}
\centering
\includegraphics[width = 10cm]{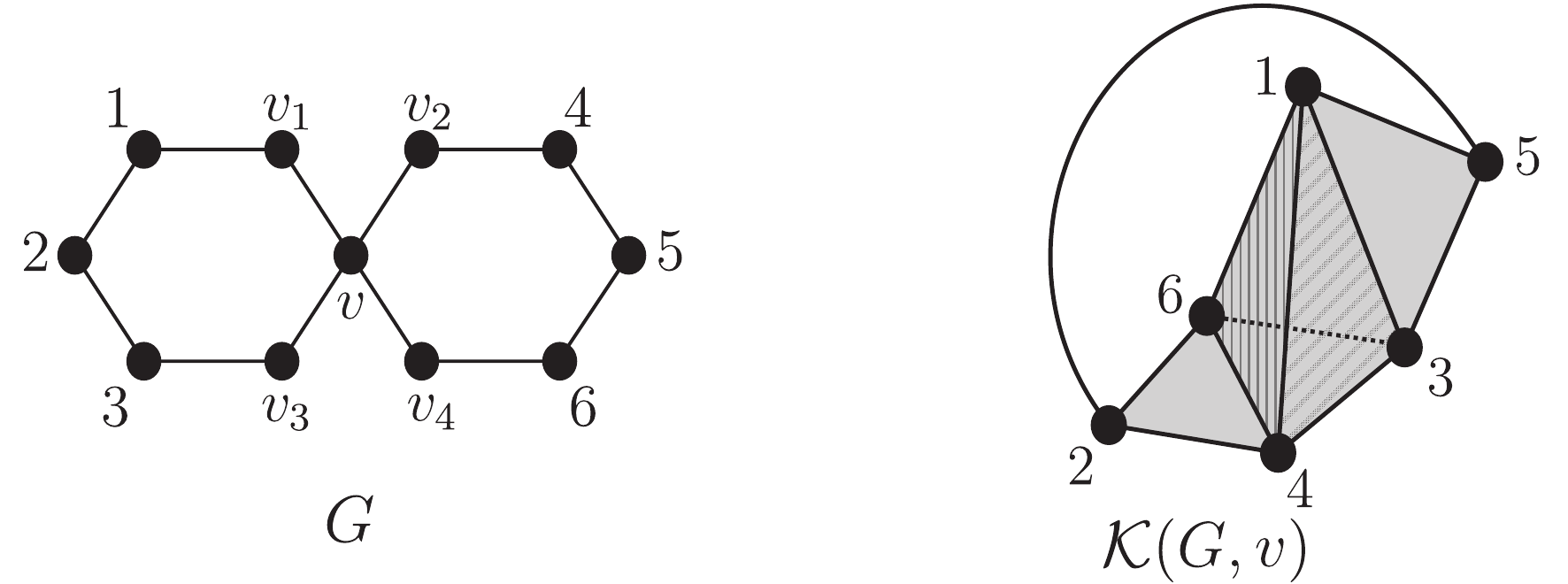}
\caption{\small{The graph $G$, and the complex $\mathcal{K}(G,v)$. Note that the vertices $1, 3, 4, 6$ constitute the 2-skeleton of the tetrahedron (its interior not filled).}}
\label{figglobaltheo}
\end{figure}

Theorem \ref{globaltheo} has many interesting consequences. The following result implies results by Nagel-Reiner \cite{Nagel}, Jonsson \cite{Jonsson} and Barmak \cite{Barmak}:

\begin{corollary} Let $G$ be a loopless graph containing a vertex $v$ such that $lk(v)$ is an independent set (i.e. $v$ is not a vertex of any triangle). Then $I_G$ is homotopy equivalent to the suspension of a simplicial complex. In particular, it holds if $G$ is triangle free or bipartite.
\end{corollary}

Furthermore, there are some particular cases of Theorem \ref{globaltheo} which are specially useful when doing computations. The simplest case, when $n=1$, implies Corollary \ref{addingleave}. Next we list some other interesting cases when $n=2$.

\begin{corollary}\label{lastcor}
\item [(1)] Let $n = 2$. Then $\mathcal{K}(G,v)$ is a flag simplicial complex which can be expressed by $I_H$, with $H$ being a graph obtained from $G-st(v)$ by connecting every vertex in $W_1$ with every vertex in $W_2$. Note that if $w \in W_1 \cap W_2$, then the graph $H$ contains a loop based in $w$. Figure \ref{figpic} illustrates this construction.
\item [(2)] Let $n=2$ and $|W_1| = |W_2| = 1$. Then we get the result of Csorba, Theorem \ref{csorba}.
\item [(3)] We get the direct generalization of Csorba result when considering the case $n=2$ and $|W_1| = 1$ ($|W_2|$ arbitrary).
This case can be thought as contracting a path of length 3, that is, for a loopless graph $G$ with two vertices $w_1$ and $v_2$ connected
by a path $L$ of length three (the two vertices between them having order 2), consider the graph $H$ obtained from $G$ by contracting $L$.
Then $I_G \sim_h \Susp I_H$. As mentioned before, $H$ is allowed to have multiple edges and loops, as shown
in Figure \ref{creatingloop}. \footnote{Given a connected graph $G$, let $G_3$ be the graph obtained by replacing each edge of $G$ by a path of length $3$. Csorba shows that if $G$ is not a tree and has $n$ vertices and $e$ edges, then $I_{G_3}$ is homotopy equivalent to $S^{e-1} \vee S^{n-1}$ \cite{Csorba}. This result follows immediately by using Corollary \ref{lastcor}(3) $n-1$ times till one gets the wedge of $e-n+1$ triangles.}
\end{corollary}

\begin{figure}
\centering
\includegraphics[width = 10.5cm]{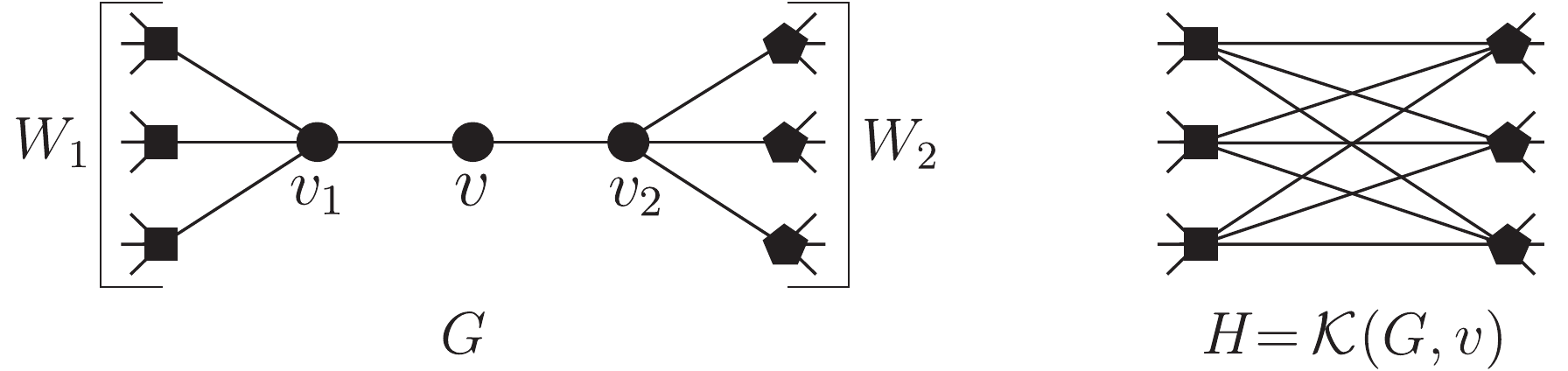}
\caption{\small{The graph $G$ and $\mathcal{K}(G,v)$ when $v$ has two adjacent vertices, as explained in Corollary \ref{lastcor}(1).}}
\label{figpic}
\end{figure}

\begin{figure}
\centering
\includegraphics[width = 9.2cm]{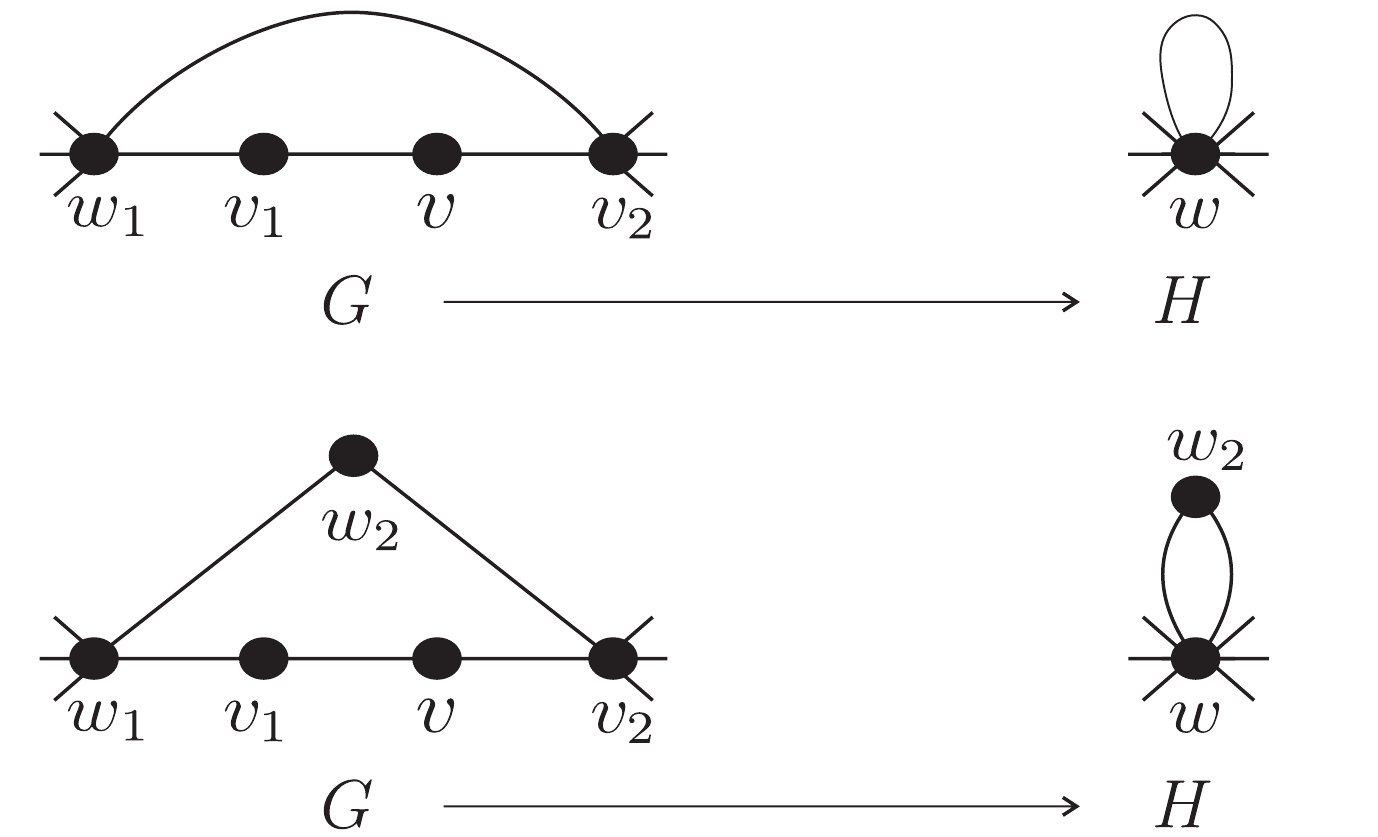}
\caption{\small{Creating loops and multiple edges using Generalized Csorba Lemma in Corollary \ref{lastcor}(3).}}
\label{creatingloop}
\end{figure}

\begin{remark}\label{rem1}
There are some natural situations in which conditions of Theorem \ref{globaltheo} are not satisfied. The simplest such example is illustrated in Figure \ref{remark1}. However, as $v_1$ dominates $v$, this case can be reduced by using Lemma \ref{dominationlemma}, yielding $I_G \sim_h I_{G-v_1} \vee \Susp I_{st(v_1)}$.
\end{remark}

\begin{figure}
\centering
\includegraphics[width = 3.2cm]{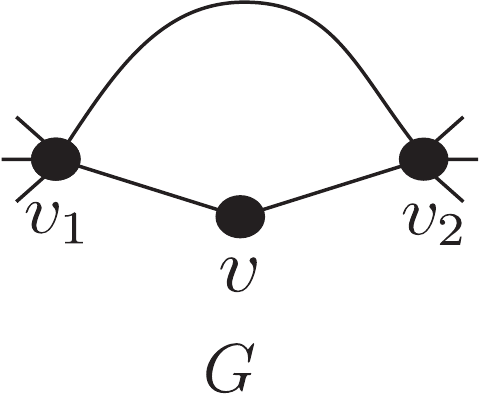}
\caption{\small{The simplest case not covered by Theorem \ref{globaltheo}.}}
\label{remark1}
\end{figure}

In Theorem \ref{globaltheo} even if one starts with a bipartite circle graph, the simplicial complex $\mathcal{K}(G,v)$ is not necessarily a flag complex. Even if one assumes that the degree of vertex $v$ is 2 (as in Corollary \ref{lastcor}) it is not clear whether the resulting simplicial complex is the independence complex of a circle graph. However, in some special situations described in Proposition \ref{propdibu} we are able to conclude so. In particular, if $H$ is obtained from $G$ by collapsing a path of length 3 as in Corollary~\ref{lastcor}(3), then if $G$ is a bipartite circle graph we conclude that $H$ is a circle graph.

\begin{proposition}\label{propdibu} 
Let $G$ be a loopless bipartite circle graph containing a degree-2 vertex $v$ with adjacent vertices $v_1$ and $v_2$. For $i=1,2$, write $W_i^{left}$, ($W_i^{right}$) for the chords intersecting $v_i$ at the left (right) side of the chord representing $v$ (see Figure \ref{wvacios}(1)). Let $H$ be the graph obtained from $G$ by applying the construction in Corollary \ref{lastcor}(1) and removing the vertices with loops in case they are created.
\begin{enumerate}
\item[(1)] If one of the sets of chords $W_1^{left}$, $W_1^{right}$, $W_2^{left}$ or $W_2^{right}$ is empty, then $H$ is a circle graph.
\item[(2)] In particular, if $G$ contains a path of length 3, then the graph $H$ obtained after its contraction is a circle graph.
\end{enumerate}
\end{proposition}

%\begin{proposition}\label{propvacio} Let $G$ be a loopless bipartite circle graph.
%\begin{enumerate}
%\item[(1)] If $G$ contains a path of length 3, then the graph $H$ obtained after its contraction is a circle graph.
%\item[(2)] More generally, assume that $G$ contains a degree-2 vertex $v$ with adjacent vertices $v_1$ and $v_2$, and write $W_i^{left}$ ($W_i^{right}$) for the chords intersecting $v_i$ at the left (right) side of the chord representing $v$, with $i=1,2$ (see Figure \ref{wvacios}(1)). If one of the sets of chords $W_1^{left}$, $W_1^{right}$, $W_2^{left}$ or $W_2^{right}$ is empty, then the graph $H$ obtained from $G$ by applying the construction in Corollary \ref{lastcor}(1) and removing the vertices with loops, is a circle graph.
%\end{enumerate}
%\end{proposition}

\begin{proof} The proof is essentially given in Figure \ref{wvacios}. As usual, we keep the name of a vertex for its associated chord. In the chord diagram associated to $G$, draw $v$ as a vertical chord, and $v_1$ and $v_2$ as two chords on the top and on the bottom. Assume that $W_2^{right} = \emptyset$. To construct the bipartite chord diagram associated to $H$, just remove the chords $v$, $v_1$ and $v_2$, and reconnect properly the chords of $W_1^{left}$ and $W_2^{left}$ as shown in Figure \ref{wvacios}(2) to reach a chord diagram of $H$. It is not hard to see that the chords intersecting neither $v_1$ nor $v_2$ can be arranged properly so they keep their adjacencies in $H$.

\begin{figure}
\centering
\includegraphics[width = 12cm]{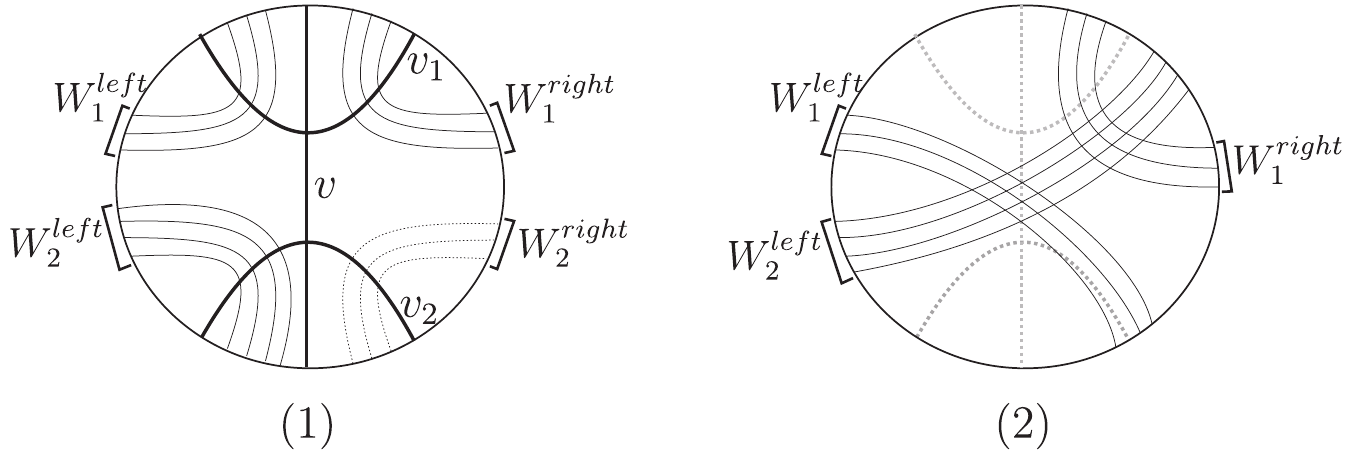}
\caption{\small{Proposition \ref{propdibu} is illustrated by showing a scheme of the chord diagrams associated to $G$ and $H$.}}
\label{wvacios}
\end{figure}
\end{proof}

\section{Applications to Knot Theory}\label{sec7}

\subsection{Introduction to (extreme) Khovanov homology} \label{seckhov}

Khovanov homology is a powerful link invariant introduced by Mikhail Khovanov at the end of last century \cite{Khovanov}. More precisely, given a link $L$, he constructed a collection of bigraded groups $H^{i,j}(L)$ arising as the homology groups of certain chain complexes, in such a way that
$$
J(L)(q) = \sum_{i,j} q^j (-1)^i rk(H^{i,j}(L)),
$$
where $J(L)$ is the Jones polynomial of $L$. It is known that Khovanov homology detects the unknot \cite{unknot}.

We review the definition of Khovanov homology by following the approach given by Viro \cite{Vir} and summarized in \cite{GMS}.

Given an oriented diagram $D$ of a link $L$, write $w = w(D) = p-n$ for its writhe, with $p$ and $n$ its number of positive and negative crossings (see Figure \ref{marcadores} for sign convention). A state $s$ is an assignation of a label, $A$ or $B$, to each crossing of $D$. Write $\sigma = \sigma(s) = a(s) - b(s)$, with $a(s)$ ($b(s)$) being the number of $A$ ($B$) labels in $s$. For each $s$, smooth each crossing of $D$ according to its label following Figure \ref{marcadores}. Now, enhance each of the $|s|$ circles with a sign $\epsilon = \pm 1$. We keep the letter $s$ for the enhanced state to avoid cumbersome notation. Write $\tau = \tau (s) = \sum_{i=1}^{|s|} \epsilon_i$. Then, the indices associated to the state $s$ are
$$
i = i(s)= \frac{w - \sigma}{2},     \quad \quad      j = j(s) = w + i + \tau.
$$

\begin{figure}
\centering
\includegraphics[width = 8cm]{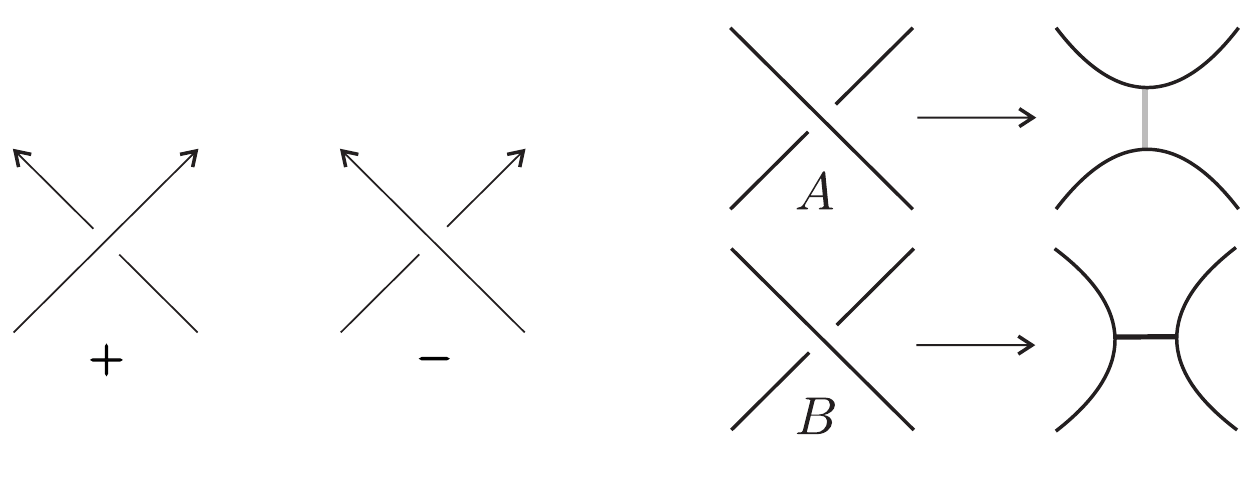}
\caption{\small{The sign convention and the smoothing of a crossing according to its $A$ or $B$-label are shown. $A$-chords ($B$-chords) are represented by light (dark) segments.}}
\label{marcadores}
\end{figure}

The enhanced state $t$ is adjacent to $s$ if they are identical except in the neighborhood of a crossing $x$, where they differ as shown in Figure~\ref{posibilidades}. In particular, this implies that $i(t) = i(s) + 1$ and $j(t) = j(s)$.\\

\begin{figure}[h]
\centering
\includegraphics[width = 10cm]{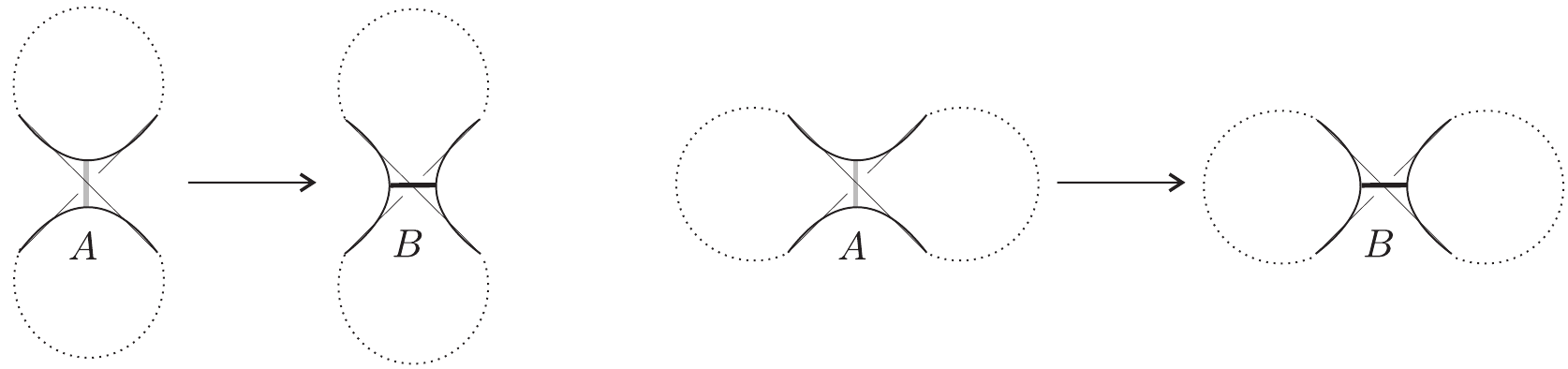}
\caption{\small{All possible enhancements when melting two circles are: $(++ \rightarrow +), \, (+- \rightarrow -), \, (-+ \rightarrow -)$. The possibilities for the splitting are: $(+ \rightarrow +-)$, \, $(+ \rightarrow -+)$ or $(- \rightarrow --)$.}}
\label{posibilidades}
\end{figure}

Let $C^{i,j}(D)$ be the free $\mathbb{Z}$-module generated by the set of enhanced states $s$ of $D$ with $i(s) = i$ and $j(s) = j$. Order the crossings in $D$. Now fix an integer $j$ and consider the ascendant complex
$$
\ldots \quad \rightarrow \quad C^{i,j}(D) \quad \stackrel{d_i}{\longrightarrow} \quad C^{i+1,j}(D) \quad \longrightarrow \quad \ldots
$$
with differential $d_i(s) = \sum (s:t) t$, where $(s:t) = 0$ if $t$ is not adjacent to $s$ and otherwise $(s:t) = (-1)^k$, with $k$ being the number of $B$-labeled crossings coming after the change crossing $x$. It turns out that $d_{i+1} \circ d_i = 0$ and the corresponding homology groups
$$
H^{i,j}(D)=\frac{\textnormal{ker} (d_i)}{\textnormal{im}(d_{i-1})}
$$
are independent of the diagram $D$ representing the link $L$ and the ordering of its crossings, that is, these groups are link invariants. They are the Khovanov homology groups $H^{i,j}(L)$ of $L$ (\cite{Khovanov, BarNatan}).

Let $j_{\max} = j_{\max}(D) = \max \{ j(s)\ | \ s \textnormal{ is an enhanced state of } D\}$. We will refer to the complex $\{ C^{i,j_{\max}}(D), d_i\}$ as the extreme Khovanov complex, and to the corresponding groups $H^{i,j_{\max}}(D)$ as the (potential) extreme Khovanov homology groups.

We remark that the integer $j_{\max}(D)$ depends on the diagram, and may differ for two different diagrams representing the same link. Given an oriented link diagram $L$, we write $\overline{j}(L)$ for the highest value of $j$ such that $H^{i,j}(L)$ is non-trivial for at least one value of $i$. This value does not depend on the chosen diagram. We will call the corresponding groups $H^{i,\overline{j}}(L)$ real-extreme Khovanov homology groups. Note that for every diagram $D$ representing a link $L$, $j_{\max}(D) \geq \overline{j}(L)$.

There are analogous definitions for $j_{\min}(D)$ and $\underline{j}(L)$. Actually, $j_{\min}(D) \leq \underline{j}(L)$ and if $D^*$ represents the mirror image of a link diagram $D$, then $j_{\max}(D) = - j_{\min}(D^*)$. From now on we work with $j_{\max}(D)$ and $\overline{j}(L)$.

\begin{question} {\rm{\cite{GMS}}} \label{question}
Does every oriented link $L$ have a diagram $D$ whose associated $j_{\max}(D) = \overline{j}(L)$?
\end{question}

In the next subsection we show that Conjecture \ref{conj} leads to a negative answer to Question \label{question}.

\subsection{Extreme Khovanov homology as the independence complex of bipartite circle graphs}\label{subseclando}

Given a state $s$ write $sD$ for the set of circles and chords obtained when smoothing the crossings of $D$ according to the labels given by $s$.

\begin{definition} \label{deflando}
Let $D$ be a link diagram and $s_B$ the state assigning a $B$-label to every crossing of $D$. The Lando graph $G_D$ associated to $D$ is the (bipartite) circle graph associated to $s_BD$\footnote{In \cite{GMS} the Lando graph is defined as the circle graph associated to the state $s_A$, since that paper deals with the minimal-extreme Khovanov homology.}.
\end{definition}

An $A$ or $B$-chord is said to be admissible in $s$ if it has both endpoints in the same circle of $sD$. Note that $G_D$ can be thought as the disjoint union of the circle graphs arising from each of the chord diagrams in $s_BD$ after removing non-admissible chords, as they do not play any role in the construction of $G_D$. Figure~\ref{ejemplolando} exhibits a diagram $D$, the corresponding $s_BD$ and its Lando graph~$G_D$.

\begin{figure}
\centering
\includegraphics[width = 9.5cm]{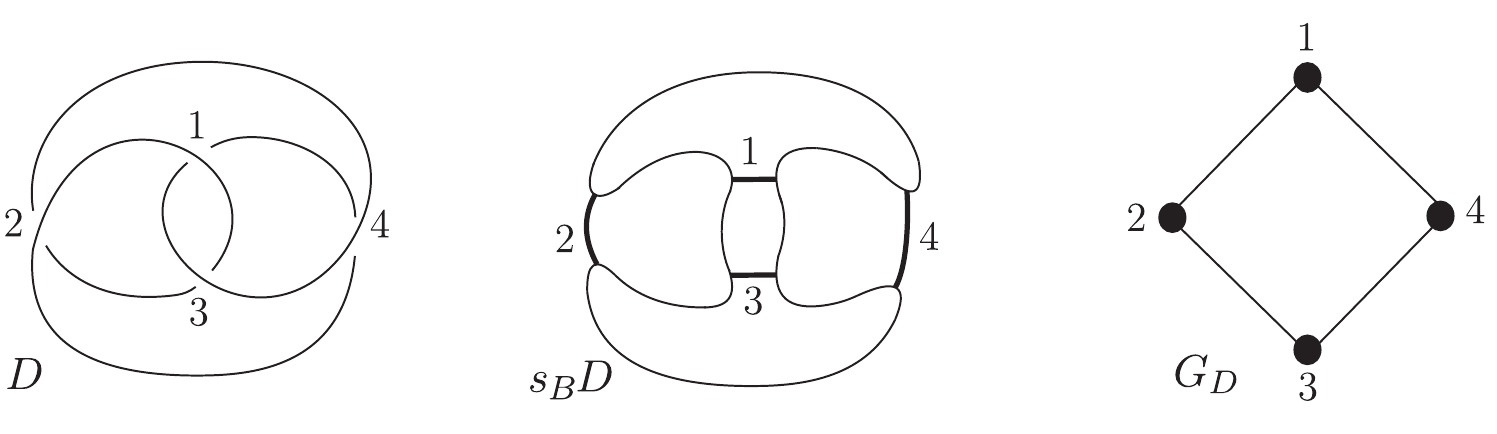}
\caption{\small{A diagram $D$ representing the trefoil knot, $s_BD$ and the corresponding Lando graph $G_D$ are shown.}}
\label{ejemplolando}
\end{figure}

Note that Lando graphs are bipartite. As not every graph is a circle graph, not every graph is a Lando graph. Recall that a graph $G$ is the Lando graph associated to some link diagram if and only if $G$ is a bipartite circle graph.

Given a chord diagram, it is not hard to reconstruct one of its associated link diagrams. One just have to replace each $B$-chord with a crossing, reversing the arrow in Figure \ref{marcadores}. However, in general given a Lando graph it is not easy to reconstruct one of its associated chord diagrams, hence link diagrams.

Let $I_{G_D}$ be the independence simplicial complex of the graph $G_D$.

\begin{definition}
Let $C_i(I_{G_D})$ be the free abelian group generated by the simplexes of $I_{G_D}$ of dimension $i$. The chain complex $\{C_i(I_{G_D}), \partial_i \}$ where $\partial_i$ is the standard differential, is called the Lando descendant complex of the link diagram $D$. It is assumed that the vertices of $G_D$ inherit the predetermined order given by the crossings of $D$.

The reduced homology groups of this chain complex are called the Lando homology groups of $D$
$$
\tilde{H}_{i}(I_{G_D})=\frac{\textnormal{ker} (\partial_i)}{\textnormal{im}(\partial_{i+1})}.
$$
\end{definition}

Following \cite{GMS}, the extreme Khovanov complex of a link diagram can be expressed in terms of the independence complex of its associated Lando graph as follows:

\begin{theorem} {\rm{\cite{GMS}}}  \label{KeyTheorem}
Let $L$ be an oriented link represented by a diagram $D$ with $p$ positive crossings. Let $G_D$ be its associated Lando graph and let $j=j_{\max}(D)$. Then the Lando descendant complex $\{C_i(I_{G_D}), \partial_i\}$ is isomorphic to the extreme Khovanov complex $\{C^{i,j}(D), d_i\}$. In particular
$$
H^{i,j}(D) \approx \tilde{H}_{p-i-1}(I_{G_D}).
$$
\end{theorem}

Note that while computing full Khovanov homology (or even Jones polynomial) is NP-hard \cite{complexity}, computing the homology of the independence complex of a graph (so potential extreme Khovanov homology) has polynomial-time complexity.

Theorem \ref{KeyTheorem} allows us to translate some open problems related to Khovanov homology into the language of circle graphs and homotopy theory. In particular, since there exist links whose real-extreme Khovanov complex is not torsion-free (the torus link $T=T(5,6)$ is such an example), Conjecture~\ref{conj} gives a negative answer to Question \ref{question}.

In the next subsections we present some consequences of our work on circle graphs related with the extreme Khovanov homology of some families of links.

\subsection{Torus knots}

\begin{lemma} Let $D$ be a link diagram.
\begin{enumerate}
\item [(1)]  If $s_BD$ contains a set of five chords depicted as in Figure \ref{csorbabraid}(1), then $D$ contains the 3-braid tangle $\sigma_1 \sigma_2 \sigma_1 \sigma_2 \sigma_1$ and $G_D$ contains a path of length 4. (See Figure \ref{torus} for braid convention.)
\item [(2)] More generally, if $s_BD$ contains a piece of $n$ chords interlaced as in Figure \ref{csorbabraid}(2), then $D$ contains the 3-braid tangle $(\sigma_1 \sigma_2)^{\frac{n-1}{2}} \sigma_1$ when $n$ is odd and $(\sigma_1 \sigma_2)^{\frac{n}{2}}$ when $n$ is even. In both cases $G_D$ contains a path of length $n-1$.
\end{enumerate}
\end{lemma}

\begin{figure}
\centering
\includegraphics[width = 11.2cm]{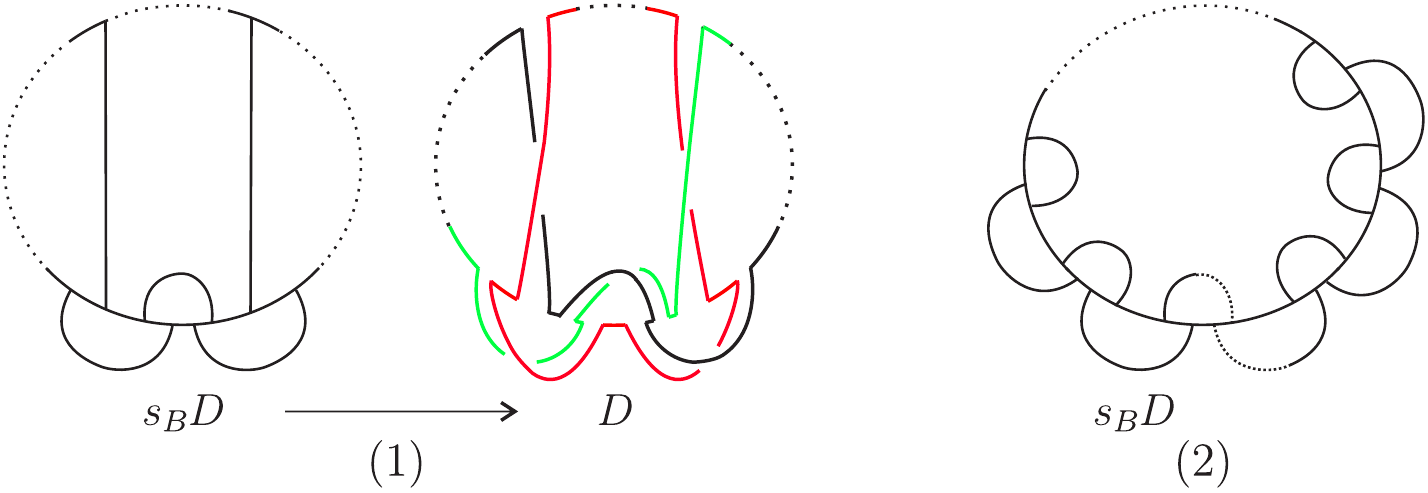}
\caption{\small{The piece depicted in the chord diagram $s_BD$ in (1) can be thought as the tangle $\sigma_1 \sigma_2 \sigma_1\sigma_2\sigma_1$ in $D$. The reconstruction of the case (2) is analogous.}}
\label{csorbabraid}
\end{figure}

\begin{proof}
The proof is illustrated in Figure \ref{csorbabraid}.
\end{proof}

\begin{figure}
\centering
\includegraphics[width = 9.2cm]{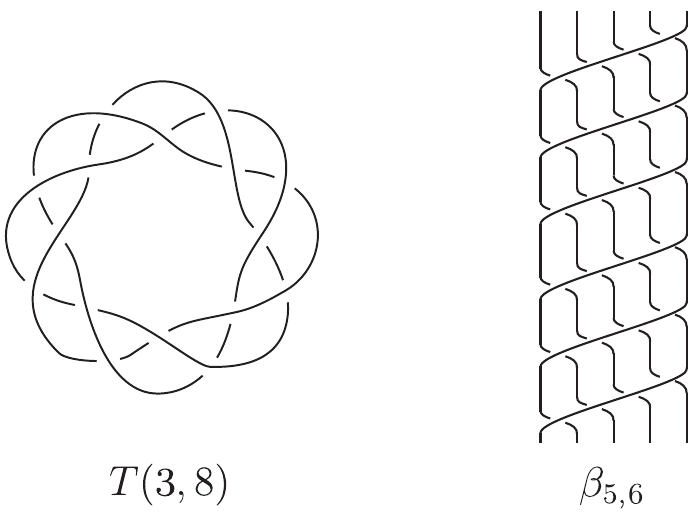}
\caption{\small{The torus link $T(3,8)$ and the braid $\beta_{5,6} = (\sigma_4 \sigma_3 \sigma_2 \sigma_1)^6$, whose closure represent the classical diagram of $T(5,6)$.}}
\label{torus}
\end{figure}

Given two positive integers $p$ and $q$, write $T(p,q)$ for the positive torus link of type $(p,q)$. Any torus link $T(p,q)$ can be represented as the closed braid on $p$ strands $\widehat{\beta}_{p,q} = ( \sigma_{p-1} \sigma_{p-2}\ldots \sigma_2 \sigma_1 )^q$. See Figure \ref{torus} for such examples and braid conventions.

\begin{corollary}\label{ciclotoro}
The Lando graph associated to the classical diagram of the torus link $T(3,q)$ (that is, the diagram depicted as the closed braid $\beta_{3,q}$) is $C_{2q}$, the cycle graph of order $2q$. See Figure \ref{torus3} for such an example.
\end{corollary}

\begin{figure}
\centering
\includegraphics[width = 11cm]{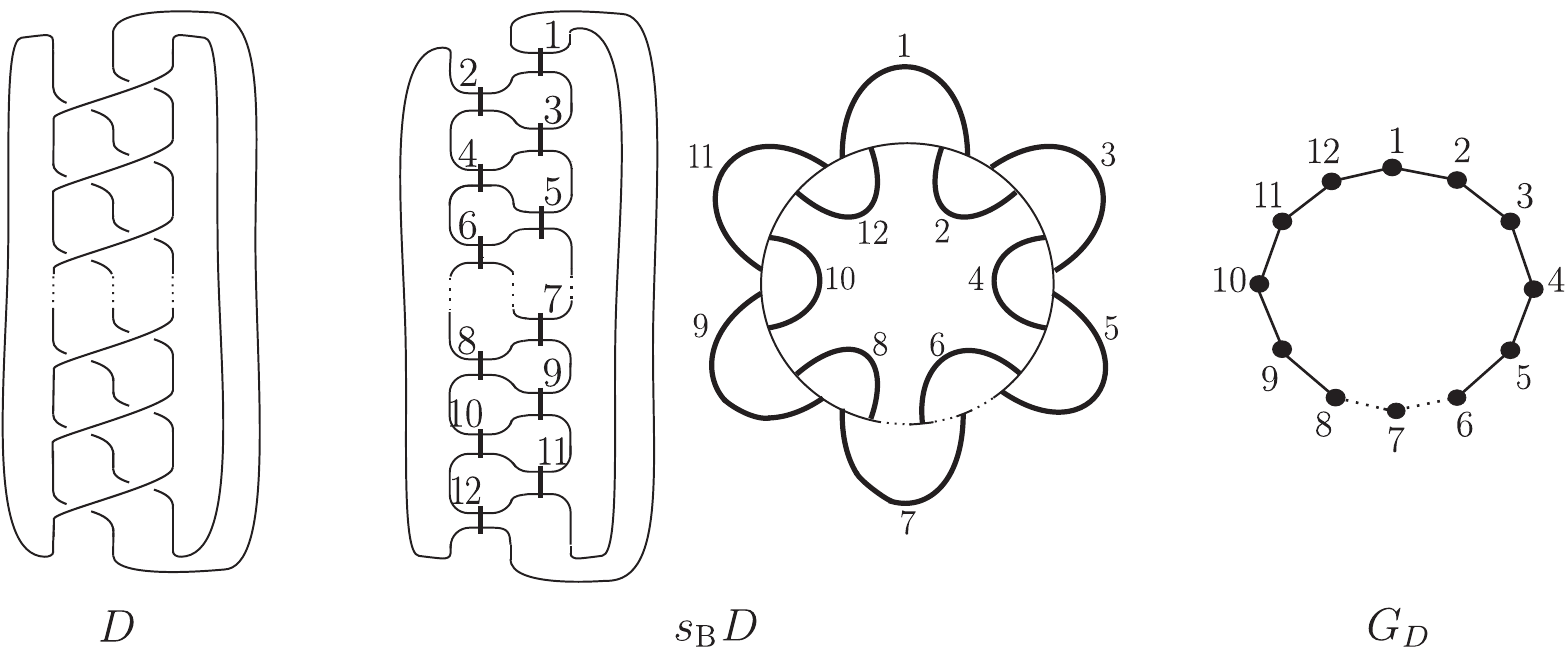}
\caption{\small{A diagram $D$ representing a torus link on 3 strands, its associated $s_BD$ and the Lando graph $G_D$ are shown. The case $k=6$ is shown when considering the dotted lines as if they were solid.}}
\label{torus3}
\end{figure}

In general, determining a closed formula for the real-extreme Khovanov homology groups of torus links, $H^{i,\overline{j}} (T(p,q))$, is an open problem. The case when $p=2$ is well known, as $T(2,q)$ are $B$-adequate, and therefore they have one non-trivial real-extreme Khovanov homology group (the associated Lando graphs $G_{\widehat{\beta}_{2,q}}$ are empty). In the next result we give an answer when $p=3$.

\begin{corollary}\label{cortor}
The real-extreme Khovanov homology groups of torus links on 3 strands are isomorphic to the reduced homology groups of the independence complex of $C_{2q}$. More precisely, $$H^{i,\overline{j}}(T(3,q)) = \tilde{H}_{2q-1-i} (I_{C_{2q}}) = \left\{
  \begin{array}{lll}
     \mathbb{Z}  &  \mbox{if } 2q = 3k \pm 1,  & i = 2q-k, \\
     \mathbb{Z} \oplus \mathbb{Z} &  \mbox{if } 2q = 3k, & i = 2q-k, \\
     0  &  \mbox{otherwise}. &
  \end{array}
  \right. $$
\end{corollary}

\begin{proof}
The first equality holds by considering the classical diagram of $T(3,q)$ given by the closed braid $\beta_{3,q}$ together with Theorem \ref{KeyTheorem} and Corollary \ref{ciclotoro}. The fact that it has $2q$ positive crossings and Proposition \ref{propcycle} complete the proof.
\end{proof}

\begin{remark}
It is worth noting that the independence complex of the graph arising from the closed braid $\widehat{\beta}_{p,q}$ representing the torus link $T(p,q)$ is contractible in many cases, e.g. when $p=4,5,6,7$ and $8$, so the associated extreme Khovanov homology groups are trivial. If Conjecture \ref{conj} holds, then those torus links having torsion in its real-extreme Khovanov homology groups cannot be represented by a diagram $D$ such that $j_{\max}(D) = \overline{j}(L)$. Examples of such knots are $T(5,6)$, $T(5,9)$ and $T(6,7)$ (according to \cite{knotatlas}) and also $T(5,11)$, $T(5,21)$, $T(6,11)$ and $T(7,8)$, checked by Alexander Shumakovitch and $T(8,9)$ checked by Lukas Lewark.

We found very interesting the case of torus links $T(4,q)$ with $q \geq 5$, for which we conjecture (based on Shumakovitch computations) that the real-extreme Khovanov homology groups are just $\mathbb{Z}_2$ (precisely $H^{2q,6q-1}(T(4,q)) = \mathbb{Z}_2$) and that the difference between potential and real-extreme Khovanov homology grades, $j_{\max}(D) - \overline{j}(L)$ for a diagram $D$ of $L$, is not bounded.\footnote{One can conjecture that for torus knots $T(p,q)$ with $p>3$, the real-extreme Khovanov homology converges to a finite abelian group when $q \to \infty$. With much less confidence we can ask whether this limit is the same as real-extreme Khovanov homology of $T(p,p+1)$ (compare with \cite{Gor, Sto, Willis} and with \cite[Conjecture 6.1]{PrzytyckiRadmila}).}
\end{remark}

\subsection{Gaps in extreme Khovanov homology} \label{hthick}

In \cite{GMS} the authors give a procedure for constructing knot diagrams having as many non-trivial extreme Khovanov homology groups as desired (that is, knots as far of being $H$-thin as one wishes); these non-trivial groups are correlative in the sense that there do not exist gaps between them. Based on Examples \ref{gap1} and \ref{gap2}, we show two families of knots whose non-trivial extreme Khovanov homology groups are separated by either one or two gaps as long as desired.

\begin{theorem}\label{theogaps} \
\begin{enumerate}
\item [(1)] For every $n > 0$ there exists an oriented knot diagram $D_n$ with exactly two non-trivial extreme Khovanov homology groups $H^{i_1,j_{max}}(D_n)$ and $H^{i_2,j_{max}}(D_n)$ such that $i_1 - i_2 = n-1$.
\item [(2)] For any $m,n > 0$  there exists an oriented knot diagram $D_{m,n}$ with exactly three non-trivial extreme Khovanov homology groups $H^{i_1,j_{max}}(D_{m,n})$, $H^{i_2,j_{max}}(D_{m,n})$ and $H^{i_3,j_{max}}(D_{m,n})$ such that $i_1 - i_2 = m-1$ and $i_2 - i_3 = n$.
\end{enumerate}
\end{theorem}

\begin{proof}
(1) Let $D_n$ be the oriented diagram shown in Figure \ref{diaggaps}. Its associated $s_BD_n$ is the chord diagram $\mathcal{C}_n^0$ depicted in Figure \ref{figgap1}, and therefore $G_{D_n}$ is the graph $G_{n}^0$ in Example \ref{gap1}, with $I_{G_{n}^0} \sim_h S^n \vee S^{2n-1}$.

Following a similar reasoning as in the proof of Corollary \ref{cortor}, Theorem~\ref{KeyTheorem} and the fact that the number of positive crossings of $D_n$ is 5n+1 lead to
$$
H^{i,j_{\max}}(D_n) = \left\{\begin{array}{lll}
     \mathbb{Z}  & & \mbox{if } i = 4n, \\
     \mathbb{Z}  & & \mbox{if } i = 3n+1, \\
     0 & & \mbox{otherwise.}
\end{array}\right.
$$

Finally, note that the number of components of $D_n$ is equal to $n+1$ if $n$ is even or $n+2$ otherwise. By following the construction in \cite[Remark 14]{GMS} the number of components of $D_n$ can be reduced to one in such a way that the resulting knot has the same real-extreme Khovanov homology groups as $D_n$, possibly with some shiftings. \\

(2) The proof is analogous to the one in the previous case by considering the diagram $D_{m,n}$ shown in Figure \ref{diaggaps}, whose associated $s_BD_{m,n}$ is the chord diagram $\mathcal{C}_{m,n}^0$ depicted in Figure \ref{figgap2}, from Example \ref{gap2}. Note that the diagram $D_{m,n}$ is not minimal, as it is equivalent to diagram $D'_{m,n}$ in Figure~\ref{diaggaps}.

\begin{figure}
\centering
\includegraphics[width = 12cm]{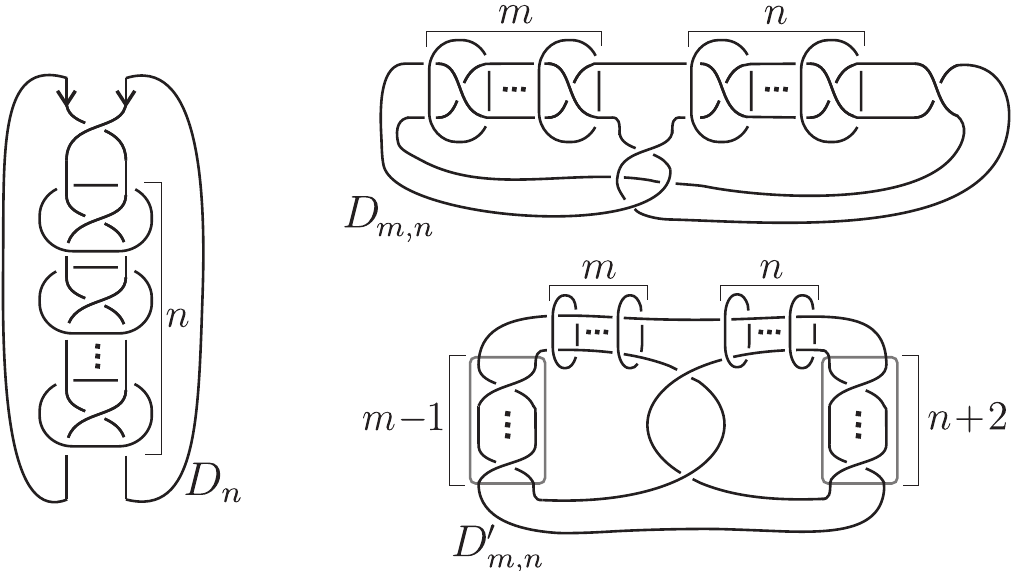}
\caption{\small{The link diagrams $D_n$, $D_{m,n}$ and $D'_{m,n}$ illustrating the proof of Theorem \ref{theogaps}.}}
\label{diaggaps}
\end{figure}

The indices of the homology groups depend on the orientation of $D_{m,n}$. Namely, if one chooses an orientation for the diagram $D_{m,n}$ such that the crossings in the blocks $[m]$ and $[n]$ are positive, then $D_{m,n}$ has $5(m+n)+2$ positive and $3$ negative crossings. From Example \ref{gap2}, $I_{G_{m,n}} \sim S^{2m+2n} \vee S^{m+2n+1} \vee S^{m+n+1}$, and therefore

$$
H^{i,j_{\max}}(D_{m,n}) = \left\{\begin{array}{lll}
     \mathbb{Z}  & & \mbox{if } i=5(m+n)+2-(2m+2n)-1=3(m+n)+1, \\
     \mathbb{Z}  & & \mbox{if } i=5(m+n)+2-(m+2n+1)-1 = 4m+3n, \\
     \mathbb{Z}  & & \mbox{if } i=5(m+n)+2-(m+n+1)-1=4(m+n), \\
     0 & & \mbox{otherwise.}
\end{array}\right.
$$

\end{proof}

For a given link $L$, the ranks of the Khovanov homology groups $H^{i,j}(L)$ can be arranged into a table with columns and rows indexed by $i$ and $j$, respectively. In Figures \ref{table1} and \ref{table2} we present the tables of the ranks and torsions of the Khovanov homology groups of the links represented by diagrams $D_5$ and $D_{3,2}$ respectively, provided kindly by Shumakovitch.

In the particular case $n=5$, according to the proof of Theorem \ref{theogaps}(1), one should get $\mathbb{Z}$ in dimensions 16 and 20, which agrees with the table in Figure \ref{table1}. In the case of $D_{m,n}$, for $m=3$ and $n=2$ one should get $\mathbb{Z}$ in dimensions 16, 18 and 20, as obtained in the table in Figure \ref{table2}.

The computations were possible by writing the links as closed braids. The 7-components link represented by $D_5$ is the closure of the 7-strands braid with 26 positive crossings  $\sigma_2\sigma_1\sigma_3\sigma_2\sigma_4\sigma_3\sigma_5\sigma_4\sigma_6\sigma_5^2\sigma_6\sigma_4\sigma_5\sigma_3\sigma_4\sigma_2\sigma_3\sigma_1\sigma_2\sigma_1^6$. The link given by $D_{3,2}$ is the closure of the 8-strands braid with 32 crossings $\sigma_5^{-1}\sigma_3\sigma_4\sigma_2\sigma_3\sigma_1\sigma_2^2\sigma_1\sigma_3\sigma_2\sigma_4\sigma_3\sigma_4^2\sigma_5^{-1}\sigma_4^4\sigma_6\sigma_5\sigma_4\sigma_7\sigma_6\sigma_5^2\sigma_6\sigma_7^{-1}\sigma_4\sigma_5\sigma_6^{-1}$.

\begin{figure}
\centering
\includegraphics[width = 12cm]{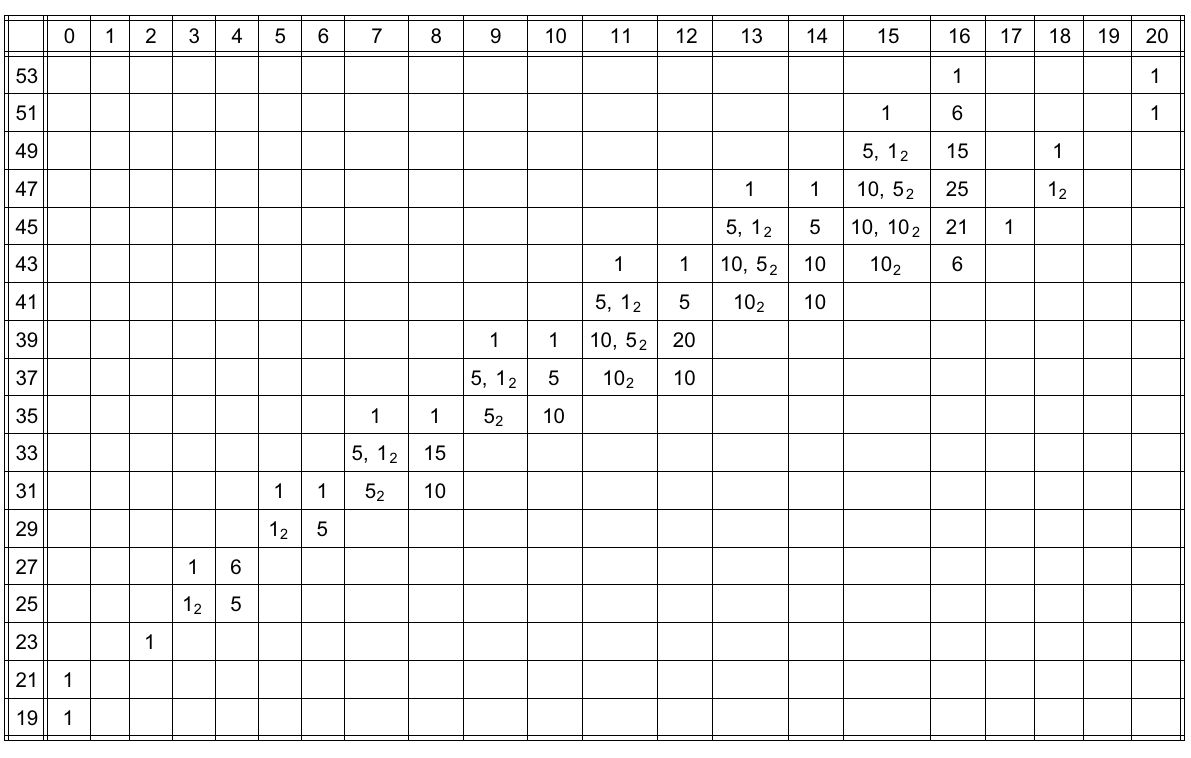}
\caption{\small{Table representing the ranks and torsions of the Khovanov homology groups of the link represented by diagram $D_5$.}}
\label{table1}
\end{figure}

\begin{figure}
\centering
\includegraphics[width = 13.1cm]{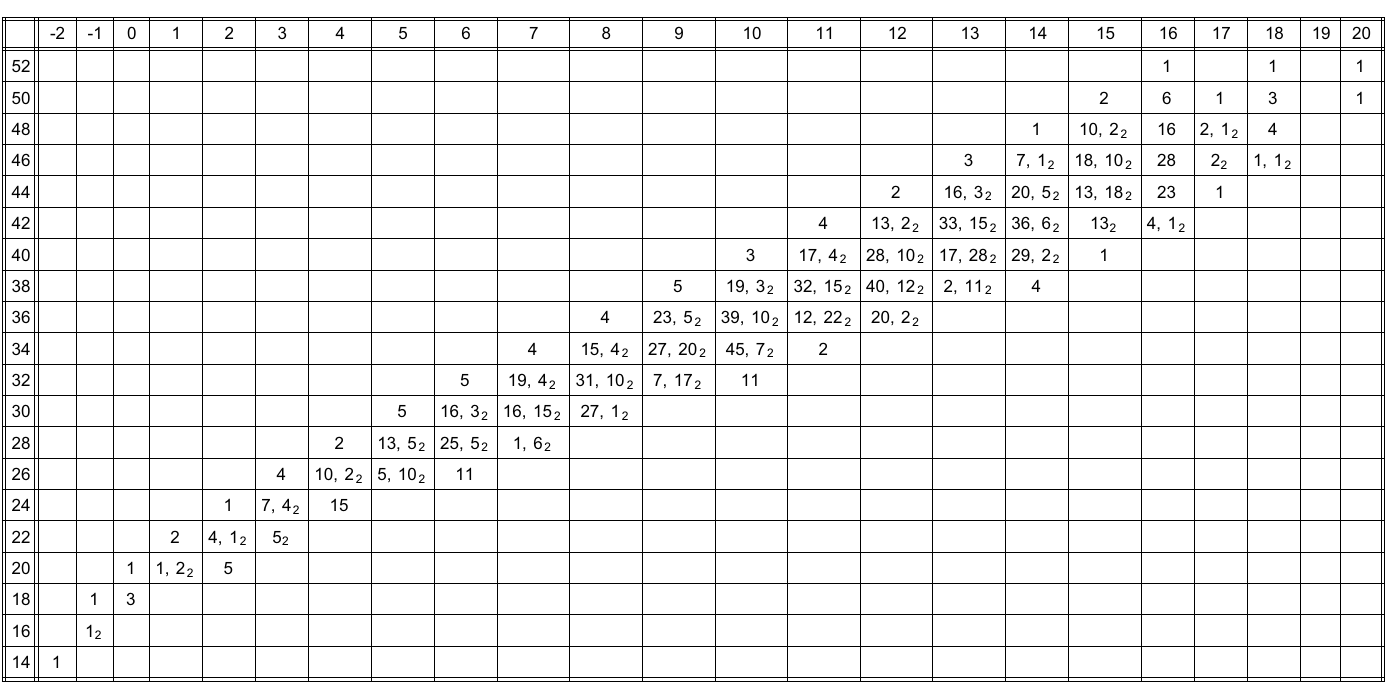}
\caption{\small{Table representing the ranks and torsions of the Khovanov homology groups of the link represented by diagram $D_{3,2}$.}}
\label{table2}
\end{figure}

\vspace{1cm}

\textbf{Aknowledgements} \, J.~H.~Przytycki was partially supported by Simons Collaboration Grant-316446, and M.~Silvero was partially supported by MTM2013-44233-P and FEDER. We would like to thank Micha{\l} Adamaszek and Victor Reiner for many useful discussions. In particular, Reiner helped us with the original version of Subsection \ref{seccone}. The authors are grateful to the Institute of Mathematics of the University of Seville (IMUS) and the Institute of Mathematics of the University of Barcelona (IMUB) for their hospitality.

\vspace{0.8cm}

\noindent J\'ozef H. Przytycki\\
Department of Mathematics\\
The George Washington University\\
University of Gda\'nsk\\
{\tt przytyck@gwu.edu}\\ \ \\

\noindent Marithania Silvero\\
Departamento de \'Algebra\\
Universidad de Sevilla\\
{\tt marithania@us.es}

\end{document}